\documentclass{article}
\usepackage[utf8]{inputenc}
\usepackage[alphabetic]{amsrefs}
\usepackage{amssymb}
\usepackage{bbm}
\usepackage{xcolor}
\usepackage{hyperref, url}
\usepackage[all]{xy}

\hypersetup{colorlinks}
\usepackage{amsthm}
\usepackage{enumitem}
\usepackage{amsmath}
\usepackage{mathtools}

\usepackage{ifthen}

\usepackage[export]{adjustbox}

\theoremstyle{plain}
\newtheorem{theorem}{Theorem}[section]
\newtheorem{lemma}[theorem]{Lemma}
\newtheorem{proposition}[theorem]{Proposition}

\newtheorem{corollary}[theorem]{Corollary}

\theoremstyle{definition}
\newtheorem{definition}[theorem]{Definition}
\newtheorem{example}[theorem]{Example}

\newtheorem{exercise}[theorem]{Exercise}
\newtheorem{problem}[theorem]{Problem}
\newtheorem{construction}[theorem]{Construction}

\theoremstyle{remark}
\newtheorem{remark}[theorem]{Remark}
\newtheorem{fact}[theorem]{Fact}
\newtheorem{convention}[theorem]{Convention}

\title{Cayley--Abels graphs: \\ Local and global perspectives}
\author{Waltraud Lederle\footnote{WL is an F.R.S.-FNRS postdoc. She was partly supported by the Early Postdoc.Mobility scholarship No. 175106 from the Swiss National Science Foundation.}}
\date{Preliminary version, comments welcome\footnote{to waltraud.lederle@uclouvain.be}}

\let\phi\varphi
\let\epsilon\varepsilon
\renewcommand\emptyset\varnothing

\newcommand{\quot}{{\setminus}}
\newcommand{\lpres}{\langle\!\langle}
\newcommand{\rpres}{\rangle\!\rangle}

\newcommand{\T}{\mathbb{T}}

\DeclareMathOperator\Aut{Aut}
\newcommand{\AutT}{\operatorname{Aut}(\T)}

\newcommand{\V}{\mathbf{V}}
\newcommand{\E}{\mathbf{E}}
\newcommand{\A}{\mathbf{A}}
\newcommand{\Or}{\mathbf{O}}
\newcommand{\D}{\mathrm{D}}
\newcommand{\N}{\mathbf{N}}
\newcommand{\B}{\mathbf{B}}
\newcommand{\val}{\mathrm{val}}
\newcommand{\U}{\mathbf{U}}

\newcommand{\foris}[3]{
	\ifthenelse{\equal{#1}{}}%
	{}%
	{#1\colon}%
	\T \setminus #2 \to \T \setminus #3}%

\newcommand{\Td}{\T_d}

\newcommand{\AutTd}{\Aut(\Td)}

\newcommand{\Sym}{\operatorname{Sym}}

\newcommand{\calE}{\mathcal{E}}

\newcommand{\ZZ}{\mathbb{Z}}
\newcommand{\RR}{\mathbb{R}}
\newcommand{\QQ}{\mathbb{Q}}
\newcommand{\NN}{\mathbb{N}}

\begin{document}
	\maketitle
	\begin{abstract}
		We give an introduction to the Cayley--Abels graph for a totally disconnected, locally compact (tdlc) group.
		It is a generalization of the Cayley graph.
		We illustrate that on the one hand, Cayley--Abels graphs are useful tools to extend concepts concerning finitely generated groups to compactly generated, tdlc groups and on the other hand, they can be used to investigate properties that, for finitely generated groups, are trivial.
	\end{abstract}

\tableofcontents

\section*{List of abbreviations and mathematical notation}

\begin{tabular}{c l}
	tdlc & totally disconnected, locally compact \\
	cgtdlc & compactly generated, totally disconnected, locally compact \\
	$\NN$ & the natural numers $\NN \coloneqq  \{0, 1, 2, 3, \dots\}$ \\
	$|X|$ & the cardinality of the set $X$ \\
	$G_\alpha$ & all elements in $G$ fixing the element $\alpha$ \\
	$G_A$ & all elements in $G$ fixing every element in the set $A$ \\
	$G_{\{A\}}$ & all elements in $G$ leaving the set $A$ invariant \\
	$\Sym(X)$ & symmetry group of the set $X$ \\
	$\V\Gamma$ & set of vertices of the graph $\Gamma$ \\
	$\A\Gamma$ & set of arcs of the graph $\Gamma$ \\
	$\E\Gamma$ & set of edges of the graph $\Gamma$ \\
	$\N\alpha$ & set of vertices neighbouring the vertex $\alpha$ \\
	$\calE\Gamma$ & set of ends of the graph $\Gamma$ \\
	$\B(\alpha,r)$ & the subgraph consisting of the closed $r$-ball around $\alpha$ \\
	$d_\Gamma$ & the graph theoretical distance on $\V\Gamma$ \\
	$\val(\alpha)$ & valency of the vertex $\alpha$ \\
	$\val(\Gamma)$ & valency of (every vertex of) the graph $\Gamma$ \\
	$\Aut(\Gamma)$ & the automorphism group of the graph $\Gamma$
\end{tabular}

\newpage

\section{Introduction}

The study of totally disconnected, locally compact (tdlc) groups takes inspiration from both Lie groups and discrete groups.
The perhaps most obvious influence from Lie groups is the very concept of a Lie group over a local field, which also comes with a Lie algebra, an exponential function, etc. (see \cite{Gloeckner2018} for an introduction).
The present text is about the discrete group flavour of tdlc groups. We study a generalization of the Cayley graph to totally disconnected, locally compact (tdlc) groups that is commonly called called the Cayley--Abels graph.

Recall that for a finitely generated group $G$ with finite, symmetric generating set $S$, the associated \emph{Cayley graph} is the graph with vertex set $G$ and such that $g,h \in G$ are adjacent if and only if there exists an $ s \in S$ such that $g = hs$.
Equivalently it is a locally finite, connected graph with an action of $G$ that is free and transitive on the vertices. 

Relaxing the condition that the group should act freely on the vertices and allowing for compact, open vertex stabilizers, it turns out that many tools from geometric group theory become available for compactly generated tdlc (cgtdlc) groups.
This brings us to the main definition of this text.

\begin{definition}\label{def:maindef}
	Let $G$ be a tdlc group. A \emph{Cayley--Abels graph} for $G$ is a locally finite, connected graph $\Gamma$ together with a vertex-transitive action of $G$ on $\Gamma$ with compact, open vertex stabilizers.
\end{definition}

Every cgtdlc group has a Cayley--Abels graph, which depends not only on the compact generating set, but also on the chosen compact, open vertex stabilizer. 
Just as in the finitely generated case, two Cayley--Abels graphs of a cgtdlc group are quasi-isometric. For this reason, several notions such as hyperbolicity or the number of ends carry over from finitely generated groups to cgtdlc groups. We will explore in Section \ref{sect:fg} which similarities between finitely generated and cgtdlc groups can be extruded from Cayley--Abels graphs.

But the Cayley--Abels graph also carries information about something that is 
trivial for discrete groups and nonexistent for Lie groups, namely arbitrarily small compact, open subgroups. Basically all research on tdlc groups crucially involves this ``local structure". These non-discrete aspects of Cayley--Abels graphs are discussed in Section \ref{sect:lc}.

Historically, to the best of the author's knowledge, Cayley--Abels graphs were introduced by Herbert Abels \cite{a73}. He was interested in Specker compactifications of a locally compact group.
More than thirty years later, Kr\"on and M\"oller \cite{KronMoller2008} brought it again to the surface, reproving some of Abels' results via graph theoretic methods but also expanding on the theory.
Nowadays, the Cayley--Abels graph often serves as a useful tool in proving theorems about tdlc groups, but is rarely the center of attention itself - something that the author of the present notes wishes to change.

\subsection{Acknowledgements}
I am grateful to R\"ognvaldur G. M\"oller for having provided his notes of a minicourse he gave in Blaubeuren; they were very helpful to me.
I also want to express my thanks to Lancelot Semal for careful reading and helpful feedback.
This survey grew out of a mini course given by the author at the WinSum School ``Locally compact groups acting on discrete structures" in December 2020 at the Bernoulli Center at the EPFL.
The author thanks Alejandra Garrido, Colin Reid, Stephan Tornier and George Willis for organizing this event.

\section{Preliminaries and Notation}
\subsection{Graphs}

Our definition of a graph largely follows Serre \cite{s03}.
The author assumes that the reader already has some familiarity with graphs and therefore keeps it short.

\begin{definition}
	A \emph{graph} is a tuple $\Gamma = (\V\Gamma,\A\Gamma,o_\Gamma,t_\Gamma,i_\Gamma)$ consisting of two sets, the \emph{vertex set} $\V\Gamma$ whose elements will be called \emph{vertices} and an \emph{arc set} $\A\Gamma$ whose elements are called \emph{arcs}, and maps $o_\Gamma,t_\Gamma \colon \A\Gamma \to \V\Gamma$ called \emph{origin} and \emph{terminus}, as well as $i_\Gamma \colon \A\Gamma \to \A\Gamma$ satisfying the following conditions.
	The map $i_\Gamma$ is a fixed-point-free involution and $o_\Gamma \circ i_\Gamma = t_\Gamma$.
\end{definition}

We will usually drop the subscript in these maps. Note that $t \circ i = o \circ i \circ i = o$.
An arc can be thought of as arrow from its origin to its terminus.

\begin{definition}
	We define a few concepts around graphs.
	\begin{enumerate}
		\item The \emph{edge set} of $\Gamma$ is $\E\Gamma \coloneqq  \{\{e,i(e)\} \mid e \in \A\Gamma \}$, its elements will be called \emph{edges}.
		\item A graph is a \emph{tree} if every sequence $(e_1,\dots,e_n)$ of arcs with $t(e_i)=o(e_{i+1})$ and $i(e_i) \neq e_{i+1}$ for every $1\leq i < n$ satisfies $o(e_1) \neq t(e_n)$.
		\item Let $n \geq 0$. A \emph{path of length $n$} from $\alpha$ to $\beta$ in $\Gamma$ is an $n$-tuple of distinct arcs $(e_1,\dots,e_n)$ with $o(e_1)=\alpha$ and $t(e_n)=\beta$ such that $t(e_{i})=o(e_{i+1})$ for $i=1,\dots,n-1$, but $o(e_{1}),\dots,o(e_n),t(e_n)$ are distinct. If there is no path of smaller length from $o(e_1)$ to $t(e_n)$ then $(e_1,\dots,e_n)$ is called a \emph{geodesic path}.
		\item A \emph{sub-path} of a path $(e_1,\dots,e_n)$ is a path $(e_i,\dots,e_j)$ with $1 \leq i \leq j \leq n$.
		\item The \emph{valency} of a vertex $\alpha \in \V\Gamma$ is the cardinality $\val(\alpha) \coloneqq  |o^{-1}(v)|=|t^{-1}(v)|$.
		If the valency of every vertex is finite, we call $\Gamma$ \emph{locally finite.}
		\item If $\val(\alpha)=\val(\beta)$ for all $\alpha,\beta \in \V\Gamma$ we call $\Gamma$ a \emph{regular graph} and define the \emph{valency} of $\Gamma$ by $\val(\Gamma) \coloneqq \val(\alpha)$.
		\item The \emph{set of neighbours} of $\alpha \in \V\Gamma$ is $\N(\alpha) \coloneqq  t(o^{-1}(\alpha)) \subset \V\Gamma$ or, in other words, those vertices that are connected to $\alpha$ by an edge.
		\item A \emph{subgraph} of $\Gamma$ is a graph $\Delta \coloneqq  (\V\Delta,\A\Delta,o_\Delta,t_\Delta,i_\Delta)$ such that $\V\Delta \subset \V\Gamma, \A\Delta \subset \A\Gamma$ and $o_\Delta,t_\Delta,i_\Delta$ are the restrictions of $o_\Gamma,t_\Gamma$ and $i_\Gamma$.
		\item Let $\Delta \subset \Gamma$ be a subgraph. The \emph{graph difference} $\Gamma \setminus \Delta$ is the minimal subgraph of $\Gamma$ with
		arc set $\A(\Gamma \setminus \Delta) \coloneqq  \A\Gamma \setminus \A\Delta$ and
		whose vertex set contains $\V\Gamma \setminus \V\Delta$. Its vertex set is $\V(\Gamma \setminus \Delta) = \{\alpha \in \V\Gamma \mid o^{-1}(\alpha) \in \A\Gamma \setminus \A\Delta \} \cup (\V\Gamma \setminus \V\Delta)$.
		\item The \emph{(graph-theoretical) distance} $d_\Gamma \colon \V\Gamma \times \V\Gamma \to \NN \cup \{\infty\}$ is defined by setting $d_\Gamma(\alpha,\beta)$ the length of a geodesic connecting $\alpha$ and $\beta$.
		\item The graph $\Gamma$ is called \emph{connected} if each two vertices have finite distance from each other, and a \emph{connected component} is a maximal, connected subgraph.
		\item Let $r \geq 0$. The \emph{ball of radius $r$ around $\alpha$} is the largest subgraph $\B(\alpha,r)$ of $\Gamma$ containing all vertices at distance at most $r$ from $\alpha$.
	\end{enumerate}
\end{definition}

Each graph also has a geometric realization as one-dimensional CW-complex, where $\V\Gamma$ is the set of $0$-cells and each edge $\{e,i(e)\}$ gives rise to a $1$-cell that is attached to $o(e)$ and $t(e)$. We will only the geometric realization in Section \ref{sect:trees}, but it is helpful to visualize a graph in that way.

\begin{example}\label{expl:graph_N_Z}
	The graph $\Gamma_\NN$ has vertex set $\NN$ and arc set $\{(n,m) \mid n,m \in \NN, \, |n-m|=1 \}$. Origin and terminus are given by $o((n,m))=n$ and $t((n,m))=m$.  In particular, the involution is $i((n,m))=(m,n)$.
	
	The graph $\Gamma_\ZZ$ has vertex set $\ZZ$ and arc set $\{(n,m) \mid n,m \in \ZZ, \, |n-m|=1 \}$. Origin and terminus are given by $o((n,m))=n$ and $t((n,m))=m$.  In particular, the involution is $i((n,m))=(m,n)$.
\end{example}

\begin{convention}\label{conv:graphs}
	In the remainder of our notes, all our graphs will be \emph{simple}, i.e. there will be no multiple edges between vertices and no loops.
	This means we can consider the arc set as a subset $\A\Gamma \subset \V\Gamma \times \V\Gamma$ not intersecting the diagonal, and $\E\Gamma \subset \{E \subset \V\Gamma \mid |E|=2 \}$.
	In particular our convention implies that $\val(\alpha) = |\N(\alpha)|$ for every vertex $\alpha \in \V\Gamma$.
	We can write paths as sequences of distict vertices.
\end{convention}


We turn now to the definition of \emph{ends} of a graph $\Gamma$, going back to Freudenthal~\cite{Freudenthal1945} and Halin~\cite{Halin1964}. A detailed introduction is \cite{Kroen05}.

\begin{definition}\label{def:ray_line}
	A \emph{ray} in a graph $\Gamma$ is an infinite sequence of distinct vertices $\rho\coloneqq (\alpha_0,\alpha_1,\alpha_2,\dots)$ such that $(\alpha_{i-1},\alpha_i) \in \A\Gamma$ for every $i \geq 1$.
	
	A \emph{line} in a graph $\Gamma$ is a bi-infinite sequence $\rho\coloneqq (\dots,\alpha_0,\alpha_1,\alpha_2,\dots)$ such that $(\alpha_{i-1},\alpha_i) \in \A\Gamma$ for every $i \in \ZZ$.
\end{definition}
	
	Then we can also talk about sub-paths and sub-rays of rays and lines: A sub-path will be a finite sequence $(\alpha_i,\dots,\alpha_j)$ with $i \leq j$ and a \emph{sub-ray} an infinite sequence $(\alpha_n,\alpha_{n+1},\dots)$.
	A ray is called a \emph{geodesic ray} if every sub-path is a geodesic path.
	Also a line is called a \emph{geodesic line} if every sub-path is a geodesic path.
	
	Two rays $\rho_1, \rho_2$ are \emph{equivalent} if there exists a third ray containing infinitely vertices from both $\rho_1$ and $\rho_2$.

\begin{definition}
	An \emph{end} of $\Gamma$ is an equivalence class of this equivalence relation.
	We denote the set of all ends by $\calE\Gamma$.
\end{definition}

\begin{exercise}
	Show that equivalence of rays is indeed an equivalence relation.
\end{exercise}

\begin{exercise}
	Determine the set of ends for a few graphs you know. If you know other types of boundaries, compare them with the set of ends.
\end{exercise}

The existence of ends is a consequence of a famous lemma by K\H{o}nig; the ``only if"-part is obvious.

\begin{lemma}[K\H{o}nig's Lemma]
	\label{lem:koenigslemma}
	A connected, locally finite graph has an end if and only if it infinite.
\end{lemma}

The set of ends $\calE \Gamma$ of a graph can be endowed with a topology that we will now define.

\begin{lemma}
	Let $\rho$ be a ray in $\Gamma$ and let $\Theta \subset \Gamma$ be a finite subgroup.
	Let $\Delta$ be a connected component of $\Gamma \setminus \Delta$.
	Then, $\Delta$ either contains finitely many vertices of $\rho$, or it contains all but finitely many vertices of $\rho$.
\end{lemma}

\begin{proof}
	Exercise.
\end{proof}

\begin{lemma}[\cite{Halin1964} (1.3)]\label{lem:ray_equivalence}
	Two rays $\rho_1, \rho_2$ in $\Gamma$ are equivalent if and only if for each finite subgraph $\Theta \subset \Gamma$, every connected component of $\Delta \setminus \Theta$ contains either all but finitely many vertices of both $\rho_1$ and $\rho_2$, or only finitely many vertices of both $\rho_1$ and $\rho_2$.
\end{lemma}

Let $\Theta \subset \Gamma$ be a finite subgraph and $\Delta$ a connected component of $\Gamma \setminus \Theta$. 
Lemma \ref{lem:ray_equivalence} allows us to talk about whether $\Delta$ \emph{contains} the end $\xi$ or not.
A basic open set of $\calE \Gamma$ consists of all ends $\xi \in \calE \Gamma$ contained in $\Delta$.

\begin{exercise}
	Prove that these sets indeed form the basis of a topology on $\calE\Gamma$.
\end{exercise}

\begin{exercise}\label{exer:ray_equiv_sufficient}
	Show that, in Lemma \ref{lem:ray_equivalence}, it is not necessary to demand that the condition holds for every finite subgraph $\Theta$, but it is enough to consider nested finite subgraphs $\Theta_1 \subset \Theta_2 \subset \dots$ covering $\Gamma$.
\end{exercise}

It is immediately obvious from the definition that each basic open set is not just open, but also closed, and its complement is a finite union of basic open sets if $\Gamma$ is locally finite and has finitely many connected components.
Lemma~\ref{lem:ray_equivalence} shows in particular that $\calE\Gamma$ is Hausdorff.

We summarize the topological properties of $\calE\Gamma$ in a lemma. We didn't prove compactness, it can be done as an exercise.

\begin{lemma}
	Let $\Gamma$ be a locally finite graph. Then $\calE\Gamma$ is compact, Hausdorff, second countable and totally disconnected.
\end{lemma}

The set of ends of a graph can be considered a ``boundary", turning the union $\V\Gamma \cup \calE\Gamma$ into a compact topological space with dense subset $\V\Gamma$. A sequence of vertices $(\alpha_n)_{n \in \NN}$ converges to an end $\xi \in \calE\Gamma$ if for every finite subgraph $\Theta \subset \Gamma$ and connected component $\Delta$ of $\Gamma \setminus \Theta$ containing $\xi$ the sequence $(\alpha_n)$ eventually lies in $\V\Delta$.

Obviously, every line $(\dots,\alpha_{-1},\alpha_0,\alpha_1,\dots)$ defines two (not necessarily different) ends: the equivalence class of $(\alpha_0,\alpha_1,\dots)$ and the equivalence class of $(\alpha_0,\alpha_{-1},\dots)$.

\begin{lemma}
	Let $\Gamma$ be a locally finite graph and let $\rho^+,\rho^- \in \calE\Gamma$ be different ends.
	Then, there exists a geodesic line $\omega$ such that $\rho^+,\rho^-$ are the two ends defined by $\omega$.
\end{lemma}

\begin{proof}
Let $\alpha \in \V\Gamma$, $n \geq 0$ and denote by $\Theta^+_n$ and $\Theta^-_n$ the connected components of $\Gamma \setminus \B(\alpha,n)$ containing $\rho^+$ and $\rho^-$, respectively. Clearly $\Theta^+_n \supset \Theta^+_{n-1}$ for all $n \geq 1$ and $\rho_+$ is the unique end contained in all $\Theta^+_n$; the analogue statement holds for $\Theta^-_n$.

Let $\Omega_n$ be the set of all geodesic paths starting in $\V\Theta^-_n \cap \V\B(\alpha,n)$ and ending in $\V\Theta^+_n \cap \V\B(\alpha,n)$. Note that $\Omega_n$ is finite and every sub-path of a geodesic path in $\Omega_n$ is a geodesic path.
Since $\rho^+ \neq \rho^-$ there exists an $N \geq 0$ such that $\Theta^+_N \neq \Theta^-_N$ and, for each $n \geq N$, every element of $\Omega_n$ has an edge in $B(\alpha,N)$.
Let $\Omega := \bigcup_{n \geq N} \Omega_n$.
Since $\Omega$ is infinite, there is an element $\omega_1 \in \Omega_N$ contained in infinitely many paths in $\Omega$. Then, there exists an element $\omega_2$, having $\omega_1$ as sub-path, contained in infinitely many paths in $\Omega$.
Inductively, we get an infinite sequence $\omega_1,\omega_2,\dots$ of geodesics such that $\omega_n$ contains $\omega_{n-1}$ as sub-path. Using Exercise \ref{exer:ray_equiv_sufficient} we see that the unique bi-infinite line $\omega$ containing all $\omega_n$ as sub-paths satisfies the claim.
\end{proof}

\begin{exercise}
	Adapt this proof to show that every end can be represented by a geodesic ray. Note that this is also a corollary of the above lemma if the graph has at least two ends.
\end{exercise}

\begin{definition}
	An end is \emph{thick} if it has infinitely many pairwise disjoint representatives, otherwise it is called \emph{thin}.
\end{definition}

The following is a consequence of Menger's theorem, see the discussion at the beginning of Section 4 in \cite{Thomassen1993}.

\begin{lemma}\label{lem:thinends}
		An end $\rho$ of a locally finite graph $\Gamma$ is thin if and only if there exists an $n \geq 0$ and a sequence of finite subgraphs $\Theta_1, \Theta_2, \dots$ with at most $n$ edges such that for all $i \geq 1$ 
		the connected component of $\Gamma \setminus \Theta_{i-1}$ containing $\rho$
		contains both the graph $\Theta_i$ and the connected component of $\Gamma \setminus \Theta_i$ containing $\rho$.
\end{lemma}

%
%
%
%
%
%
%

\subsection{Tdlc groups}

Almost every statement about totally disconnected, locally compact groups involves more or less directly the famous theorem of Van Dantzig.

\begin{theorem}[Van Dantzig's Theorem]\label{thm:vandantzig}
	Let $G$ be a tdlc group. Every neighbourhood of the identity of $G$ contains a compact, open subgroup.
\end{theorem}

Obvious examples of tdlc groups are discrete groups.
Compact examples are (infinite) direct products of finite groups and their closed subgroups, the \emph{profinite groups}.
Subgroups, products, extensions and direct limits of tdlc groups are again tdlc.
The class of tdlc groups obtained by starting with discrete groups and profinite groups forms and forming above operations is called the class of \emph{elementary groups}, a class with its own structure theory largely due to Wesolek \cite{Wesolek15}.
In contrast to the elementary groups stands the class of compactly generated, topologically simple groups, studied most notably by Caprace--Reid--Willis \cite{crw17a}.
Examples there can be found among automorphism groups of locally finite trees, Lie groups over $\QQ_p$, Neretin's group and its relatives, Kac--Moody groups and others.

\begin{exercise}
	Show that every totally disconnected topological group is Hausdorff.
\end{exercise}

In these notes we will give a proof that every cgtdlc group, after taking the quotient by a compact, normal subgroup, is (topologically) isomorphic to a group of automorphisms of a locally finite graph. Before we define graph automorphisms in the next subsection we introduce the appropriate topology. We turn the group $\Sym(X)$ together with the permutation topology into a topological group.

\begin{definition}\label{def:permtop}
	Let $X$ be a set and $\Sym(X)$ its symmetry group. The \emph{permutation topology} on $\Sym(X)$ can be defined in the following equivalent ways:
	\begin{enumerate}
		\item as the \emph{topology of pointwise convergence}, where $X$ is endowed with the discrete topology,
		\item as the \emph{compact-open topology}, where $X$ is endowed with the discrete topology,
		\item by stipulating that a basis of neighborhoods of the identity element is the set $\{\Sym(X)_{A} \mid A \subset X \text{ finite}  \}$ consisting of all fixators of finite sets.
	\end{enumerate}
\end{definition}

To understand the third way of defining the topology, note that to define a topology on a topological group it is enough to give a neighbourhood basis of the identity element, because this neighbourhood basis can then be translated everywhere by right and left multiplication.
This is a common way to define a topology on a group, but some care is needed to prove that what we get out is actually a group topology. Often the following Bourbaki lemma comes in handy. Recall that a \emph{filter} on a set $X$ is a subset of the power set $\mathcal{B} \subset \mathcal{P}(X)$, with $X \in \mathcal{B}$, that is closed under taking finite intersections and supsets.

\begin{lemma}[\cite{b98}, Ch. III, Sect. I, Subsect. 2, Prop. 1] \label{lem:bourbaki}
	Let $G$ be a group and $\mathcal B$ be a filter on $G$ satisfying the following three conditions.
	\begin{enumerate}
		\item For every $U \in \mathcal B$ there exists a $V \in \mathcal B$ such that $VV \subset U$.
		\item For every $U \in \mathcal B$ holds $U^{-1} \in \mathcal B$.
		\item For every $g \in G$ and every $V \in \mathcal B$ holds $gVg^{-1} \in \mathcal B$.
	\end{enumerate}
	Then, there exists a unique group topology on $G$ such that $\mathcal B$ is a neighbourhood basis of the identity element.
\end{lemma}

\begin{exercise}
	Show that, in Lemma \ref{lem:bourbaki}, the conditions imply that every element of $\mathcal B$ contains the identity.
\end{exercise}

\begin{exercise}
	Show that the three definitions in Definition \ref{def:permtop} are equivalent.
	Use Lemma \ref{lem:bourbaki} to show that it indeed gives a group topology on $\Sym(X)$.
\end{exercise}

\begin{proposition}\label{prop:SymXtdlc}
	Let $X$ be a set.
	\begin{enumerate}
		\item The symmetry group $\Sym(X)$ with the permutation topology is totally disconnected.
		\item Let $x \in X$ and $G \leq \Sym(X)$ closed. Assume that for every $y \in Y$ the orbit $G_x y$ is finite. Then $G_x$ is compact.
		
		In particular, if the stabilizer of every element has only finite orbits, the group $G$ with the permutation topology is a tdlc group.
	\end{enumerate}
\end{proposition}

\begin{proof}
	We first prove 1. It is enough to show that for all $g,h \in \Sym(X)$ with $g \neq h$ there exists a clopen set $A \subset \Sym(X)$ with $g \in A$ and $h \notin A$. We can assume without loss of generality that $g=1$ and $h \neq 1$. Take $x \in X$ with $hx \neq x$.
	Note that $A = \Sym(X)_x$ clearly satisfies $1 \in A$ and $h \notin A$.
	By the definition of the permutation topology $A$ is open. It is also closed, since open subgroups of topological groups are automatically closed.
	
	To prove 2., let $\{X_i \subset X \mid i \in I\}$ be the set of orbits of $G_x$.
	The product $\prod_{i \in I} \Sym(X_i)$ is a product of finite groups and thus compact by Tychonoff's theorem.
	There is an obvious group homomorphism $\prod_{i \in I} \Sym(X_i) \to \Sym(X)$.
	Check that it is continuous.
	Therefore, it has compact image, and also $G_x$, being a closed subgroup of this image, is compact.
\end{proof}

\subsection{Groups acting on graphs}

Graphs come with natural structure-preserving maps.

\begin{definition}
	Let $\Gamma$ and $\Delta$ be two graphs. A \emph{graph morphism} $\varphi \colon \Gamma \to \Delta$ consists of two maps $\varphi_\V \colon \V\Gamma \to \V\Delta$ and $\varphi_\A \colon \A\Gamma \to \A\Delta$ such that $o_\Delta \circ \varphi_\A = \varphi_\V \circ o_\Gamma$, $\phi_\V \circ t_\Gamma = t_\Delta \circ \phi_\A$ and $\phi_\A \circ i_\Gamma = i_\Delta \circ \phi_\A$.
	
	We call $\phi$ a \emph{graph isomorphism} if both $\phi_\V$ and $\phi_\A$ are bijections. We call it a \emph{graph automorphism} if it is a graph isomorphism and $\Gamma = \Delta$. The set of all graph automorhpisms of a graph $\Gamma$ is a group under composition, we denote it by $\Aut(\Gamma)$.
\end{definition}

\begin{exercise}
	Find out which of the conditions on $\phi_\V$ and $\phi_\A$ follow from the other ones.
\end{exercise}

We usually omit the subscripts and write only $\phi$ for $\phi_\V$ and $\phi_\A$.
In these notes, since by convention graphs are simple, a graph morphism $\varphi$ is uniquely determined by $\phi_\V$.

\begin{exercise}
	Find an example illustrating that a graph morphism $\phi$ is not uniquely determined by $\phi_\A$.
\end{exercise}

\begin{exercise}
	Check that a graph morphism is a graph isomorphism if and only if it has a two-sided inverse.
\end{exercise}

A famous theorem by Tits \cite{Tits1970} states that every automorphism of a tree fixes a vertex, inverts an edge or translates along a bi-infinite line. There is a similar statement for graphs.

\begin{proposition}[\cite{Halin73}]
	\label{prop:types_of_elements}
	Let $\Gamma$ be a connected graph and $\phi \in \Aut(\Gamma)$.
	Then, $\phi$ is either
	\begin{enumerate}
		\item \emph{elliptic}, i.e. there is a finite subgraph $\Theta$ of $\Gamma$ with $\phi(\Theta) = \Theta$;
		\item \emph{parabolic}, i.e. it is not elliptic and it fixes a unique thick end; or
		\item \emph{hyperbolic}, i.e. it is not elliptic and it fixes exactly two thin ends. 
	\end{enumerate}
  If $\phi$ is parabolic or hyperbolic, it has power that translates along an infinite line.
\end{proposition}

We now turn to groups acting on graphs.

\begin{definition}
	Let $G$ be a group and $\Gamma$ a graph. An \emph{action} of $G$ on $\Gamma$ is a homomorphism $\Phi \colon G \to \Aut(\Gamma)$.
\end{definition}

We adopt the usual notations from group actions. For example, for every vertex $\alpha \in \V\Gamma$ and a group element $g\in G$ we denote $g\alpha \coloneqq \Phi(g)(\alpha)$ and $G\alpha$ is then the orbit of $\alpha$ under $G$.
The stabilizer of $\alpha$ in $G$ is denoted by $G_\alpha$.
 
\begin{lemma}\label{lem:vertextrans}
	Let $G$ be a group acting on a connected graph $\Gamma$. Assume there exists a vertex $\alpha$ of $\Gamma$ such that $\N(\alpha) \subset G\alpha$, i.e. the orbit of $\alpha$ contains all neighbours of $\alpha$. Then, $G$ acts vertex-transitively on $\Gamma$.
\end{lemma}

\begin{proof}
	Exercise.
\end{proof}

The existence of a vertex-transitive action strongly restricts the possible type of the end space. The following theorem talks about the compact topological space $\V\Gamma \cup \calE\Gamma$.

\begin{theorem}[\cite{a73} Korollar 6.6]
	Let $\Gamma$ be a connected, locally finite graph. Assume that every end of $\Gamma$ is an accumulation point of one, and hence every, vertex orbit of $\Aut(\Gamma)$.
	Then, either $\calE\Gamma$ consists of at most two points, or it is homeomorphic to the Cantor set.
\end{theorem}

Recall that the graph-theoretical distance on the vertices of a graph $\Gamma$ takes values in the natural numbers, hence it turns $\V\Gamma$ into a discrete topological space.

\begin{lemma}
	Let $G$ be a topological group acting on a graph $\Gamma$.
	The following are equivalent.
	\begin{enumerate}
		\item The action of $G$ on $\V\Gamma$ is continuous.
		\item For every vertex $\alpha \in \V\Gamma$ the stabilizer $G_\alpha$ is open.
		\item The homomorphism $G \to \Aut(\Gamma)$ is continuous.
	\end{enumerate}
\end{lemma}

\begin{proof}
	We first prove that 1. implies 2.
	Assume that the map $G \times \V\Gamma \to \V\Gamma$ is continuous.
	Let $\alpha \in \V\Gamma$.
	The map $G \to \V\Gamma, \alpha \mapsto g\alpha$ has to be continuous as well, so the pre-image $G_\alpha$ of $\{\alpha\}$ is open.
	
	Now we prove that 2. implies 3.
	Note that the stabilizers $\Aut(\Gamma)_\alpha$ form a neighborhood sub-basis of the identity in $\Aut(\Gamma)$. Clearly $G_\alpha$ is the pre-image of such a subbasis-element, which proves this implication.
	
	Lastly we prove that 3. implies 1.
	We have to show that for every $\alpha \in \V\Gamma$ the set $\{(g,\beta) \mid g\beta = \alpha\}$ is open.
	But $\{(g,\beta) \mid g\beta = \alpha\} = 
    \bigcup_{g \in G} g G_{g^{-1}\alpha} \times \{g^{-1} \alpha\} $
a union of open sets.
	This finishes the proof.
\end{proof}

\begin{lemma}\label{lem:AutGamma_tdlc}
	Let $\Gamma$ be a graph. Show that $\Aut(\Gamma) \leq \Sym(\V\Gamma)$ is a closed subgroup. In particular, if $\Gamma$ is locally finite, it is tdlc.
\end{lemma}

\begin{proof}
	We show that the complement of $\Aut(\Gamma)$ is open.
	Let $g \in \Sym(\V\Gamma) \setminus \Aut(\Gamma)$. This means that there exists an arc $e \in \A\Gamma$ such that $g(e) \notin \A\Gamma$.
	Then, $g G_e = \{h \in G \mid h(e)=g(e) \}$ is an open neighbourhood of $g$ with $g G_e \cap \Aut(\Gamma) = \emptyset$.
	
	The last statement follows from Proposition \ref{prop:SymXtdlc}.
\end{proof}

\subsection{The local action}

Let $\Gamma$ be a graph and let $G$ be a group that acts vertex-transitively on $\Gamma$ of valency $d$.
Let $\alpha \in \V\Gamma$ and $g \in G_\alpha$. Then $g$ leaves $\N(\alpha)$, the set of neighbours of $\alpha$, invariant and we can see $g$ as an element, and $G_\alpha$ as a subgroup of, a symmetric group of degree $\val(\alpha)$. 

\begin{definition}
	Let $G$ be a group acting on a graph $\Gamma$. Let $\alpha \in \V\Gamma$. The \emph{local action} of $G$ at $\alpha$ is the subgroup $G_\alpha/(G_\alpha \cap G_{\N(\alpha)}) \leq \Sym(\N(\alpha))$.
\end{definition}

If $G$ acts vertex-transitively on $\Gamma$ and $\alpha,\beta \in \V\Gamma$, the local actions at $\alpha$ and $\beta$ are conjugate via any element $g \in G$ with $g\alpha = \beta$.
Now we can fix a bijection $\N(\alpha) \to \{1,\dots,d\}$ and see that the conjugacy class of $G_\alpha/(G_\alpha \cap G_{\N(\alpha)}) < \Sym(d)$ does not depend on $\alpha$.

\begin{definition}
	Let $G$ be a group acting vertex-transitively on a graph $\Gamma$ of valency $d$. The \emph{local action} of $G$ is the conjugacy class of $G_\alpha/(G_\alpha \cap G_{\N(\alpha)})$, seen as subgroup of $\Sym(d)$.
\end{definition}

%

Typically we will say that the local action is a subgroup of $\Sym(d)$ without referring to the conjugacy class.

\begin{exercise}\label{exer:local_action_normal_subgp}
	Let $N \triangleleft G$. Note that $N$ does not necessarily act vertex-transitively on $\Gamma$.
	Show that the \emph{local action} of $N$ on $\Gamma$ is still a well-defined notion and is a normal subgroup of the local action of $G$.
\end{exercise}
%

We will see later that the possible local actions of $G$ acting on various graphs $\Gamma$ can give us local information on $G$.

\begin{definition}
	Let $G$ act vertex-transitively on a graph $\Gamma$. We say that the action is \emph{locally transitive} if the local action is transitive.
\end{definition}

\paragraph{Burger--Mozes universal groups.}
We refer to \cite{GGT18} for an introduction into these groups that were originally defined by Burger--Mozes \cite{bm00a}. The basic idea is to define a group that, for a given local action, consists of all tree automorphisms defined by this local action.
Let $\T_d$ be a regular tree of degree $d$. Let $F \leq \Sym(d)$.
Choose a \emph{regular legal colouring} of $\T_d$, which means take $d$ colors and give every edge a color such that adjacent to every vertex all colors are present.
Formally, it is a map $\sigma \colon \A\T_d \to \{1,\dots,d\}$ satisfying the following conditions. For each vertex $\alpha$ the restriction $\sigma|_{o^{-1}(\alpha)} \colon o^{-1}(\alpha) \to \{1,\dots,d\}$ is a bijection, and $\sigma(e)=\sigma(i(e))$ for every $e \in \A\T_d$.
Note that every tree automorphism at every vertex induces an element of $\Sym(d)$ in an obvious way.
The \emph{Burger--Mozes universal group} $\U(F)$ is the set of all automorphisms of $\Td$ such that at every vertex this induced permutation lies in $F$. Formally,
it is defined by $$\U(F) \coloneqq  \{g \in \AutTd \mid \forall \alpha \in \V\Td \colon \, \sigma|_{o^{-1}(g\alpha)} \circ g|_{o^{-1}(\alpha)} \circ (\sigma|_{o^{-1}(\alpha)})^{-1} \in F \}.$$
The group $\U(F)$ depends on the choice of $\sigma$ only up to conjugation by an element of $\AutTd$.

\begin{exercise}
	Show that $\U(F) \leq \AutTd$ is vertex-transitive and closed.
\end{exercise}

\subsection{Quotient graphs}

In our definition of a quotient graph, we strongly make use of the convention that all our graphs are simple.

\begin{definition}
	Let $\Gamma$ be a graph and $\sim$ an equivalence relation on $\V\Gamma$.
	The \emph{quotient graph} $\Gamma/\sim$ has vertex set $\V\Gamma/\sim$, the set of all $\sim$-equivalence classes. The pair
	$([\alpha],[\beta])$, with $\alpha,\beta \in \V\Gamma$, is an arc in $\Gamma/\sim$ if and only if $\alpha \nsim \beta$ and there exists $\alpha'\sim \alpha$ and $\beta' \sim \beta$ such that $(\alpha,\beta) \in \A\Gamma$.
\end{definition}

It is obvious from the definition that the projection $\pi_\V \colon \V\Gamma \to \V(\Gamma/\sim)$ defines a graph homomorphism.

\begin{exercise}
	Give an example of a graph $\Gamma$ and an equivalence relation $\sim$ on $\Gamma$ such that there exists $\alpha \in \V\Gamma$ with $\val(\alpha) < \val([\alpha])$.
\end{exercise}

\begin{definition}
	Let $X$ be a set and $G$ a group acting on $X$.
	An equivalence relation $\sim$ on $X$ is called a \emph{$G$-congruence}
	if for all $x,y \in X$ and for all $g \in G$ we have that $x\sim y$ if and only if $gx\sim gy$.
\end{definition}

\begin{lemma}\label{lem:quotientgraph}
	Let $\Gamma$ be a graph, $G$ a group acting on $\Gamma$ and $\sim$ a $G$-congruence on $\V\Gamma$.
	Then, the action of $G$ on $\Gamma$ descends to an action of $G$ on $\Gamma/\sim$.
	
	If, in addition, $G$ is a topological group acting continuously on $\Gamma$ and the setwise stabilizer of every equivalence class is open in $G$, then $G$ acts continuously on $\Gamma/\sim$.
\end{lemma}

\begin{example}\label{ex:quotientgraph_by_normal_subgroup}
	Let $N \triangleleft G$ be a normal subgroup. Then, the orbits of the action of $N$ on $\V\Gamma$ form a $G$-congruence. To see this, we have to prove that for every $\alpha, \beta \in \V\Gamma$ and $n \in N$ with $\beta = n \alpha$ and every $g \in G$ there exists an $n' \in N$ such that $g \beta = n' g \alpha$. But this follows quickly from normality.
	We denote the quotient graph by $N \quot \Gamma$.
\end{example}

The following lemma summarizes some basic facts about the quotient graph by a normal subgroup. Its proof is straight forward; the only non-obvious statement is Part 4. This part is the easiest understood by visualizing it with $\Gamma$ a bi-infinite line and $N$ the group generated by a translation of length $2$.

\begin{lemma}\label{lem:quotient_by_normal}
	Let $G$ be a group acting on a graph $\Gamma$ and let $N \triangleleft G$ be a normal subgroup. Let $K$ be the kernel of the action of $G$ on $\Gamma$.
	\begin{enumerate}
		\item The kernel of the action of $G$ on $N\quot \Gamma$ is the subgroup $NK$.
		\item The quotient group $G/N$ acts on the quotient graph $N\quot \Gamma$.
		
		      If $G$ acts continuously on $\Gamma$, also the action of $G/N$ on $N \quot \Gamma$ is continuous.
		\item The vertex stabilizers are $G_{N\alpha} = G_\alpha N$ and $(G/N)_{N\alpha} = G_\alpha N/N$.
		\item For every $\alpha \in \V\Gamma$ we have $\val(N\alpha) \leq \val(\alpha)$.
		      If $\val(\alpha)$ is finite, equality holds if and only if
		      for every $\beta \in \V\B(\alpha,1)$ we have $N\beta \cap \V\B(\alpha,1) =\{\beta\}$.
%
	\end{enumerate}
\end{lemma}

\subsection{Quasi-isometries}

Two metric spaces are quasi-isometric if they are the same up to stretching with a finite factor and finite tearing.

\begin{definition}\label{def:qi}
	Let $(X,d_X)$ and $(Y,d_Y)$ be metric spaces. A \emph{quasi-isometry} from $X$ to $Y$ is a map $f \colon X \to Y$ such that there exist constants $K_1,\dots,K_5 >0$ with
	\begin{enumerate}
		\item $K_1 d_X(x_1,x_2) - K_2 \leq d_Y(f(x_1),f(x_2)) \leq K_3 d_X(x_1,x_2) + K_4$ for all $x_1,x_2 \in X$, and
		\item for all $y \in Y$ there exists an $x \in X$ with $d_Y(f(x),y) \leq K_5$.
	\end{enumerate}
If such an $f$ exists, then $X$ and $Y$ are called \emph{quasi-isometric}.
\end{definition}

The second condition means that $f$ is ``almost surjective''.
Quasi-isometry is an equivalence relation on metric spaces (neglecting the problem that metric spaces do not form a set).

\begin{example}
	The embedding $\ZZ \to \RR$ is a quasi-isometry with $K_1=K_3=K_5=1$ and $K_2=K_4=0$.
\end{example}

The following is obvious from the definition.

\begin{lemma}\label{lem:qi_boundedsets}
	Let $f \colon X \to Y$ be a quasi-isometry between two metric spaces.
	\begin{enumerate}
		\item Let $A \subset X$ be a subset. Then, $A$ is bounded if and only if $f(A)$ is bounded.
		\item Let $B \subset Y$ be a subset. Then, $B$ is bounded if and only if $f^{-1}(B)$ is bounded.
	\end{enumerate}
\end{lemma}

We call two graphs quasi-isometric if there is a quasi-isometry between their vertex sets, where the metric is the graph-theoretical distance.

\begin{lemma}\label{lem:NZnotqi}
The subsets $\NN \subset \RR$ and $\ZZ \subset \RR$ are not quasi-isometric.
\end{lemma}

\begin{proof}
	Let $f \colon \NN \to \ZZ$ be a quasi-isometry.
	By definition, there exists an integer $K > 0$ such that $|f(n)-f(n+1)| \leq K$ for all $n \in \NN$. Denote $X_k := [kK,(k+1)K-1]$ for all $k \in \ZZ$.
	The sets $X_k$ form a disjoint cover of $\ZZ$.
	By choice of $K$, whenever $f(n) \in X_k$ for some $n \in \NN$ and $k \in \ZZ$, we know that $f(n-1),f(n+1) \in X_{k-1}\cup X_k \cup X_{k+1}$.
	By the ``almost surjectivity"-condition of Definition \ref{def:qi}, there exist sequences $(n_i)_{i \geq 0}$ and $(m_i)_{i \geq 0}$ in $\NN$ such that $(f(n_i))_{i \geq 0} \to \infty$ and $(f(m_i))_{i \geq 0} \to -\infty$. By passing to subsequences we can assume that $n_i \leq m_i \leq n_{i+1}$ for all $i \geq 0$.
	But then, for $i$ sufficiently big, there must be an element $j_i \in [n_i,m_i]$ such that $f(j_i) \in X_0$. In particular, $f^{-1}(X_0)$ is infinite and hence unbounded. This contradicts Lemma \ref{lem:qi_boundedsets}.
\end{proof}

%
%

The set of ends is invariant under quasi-isometry; we prove here a version for locally finite graphs. 

\begin{proposition}[\cite{Moeller1992}, Proposition 1]
	\label{prop:qi_graphs}
	Let $\Gamma$ and $\Delta$ be two locally finite, connected graphs and $f \colon \V\Gamma \to \V\Delta$ a quasi-isometry.
	Then, $f$ extends uniquely to a continuous map $\bar{f} \colon \V\Gamma \cup \calE \Gamma \to \V\Delta \cup \calE\Delta$, which restricts to a homeomorphism $\bar{f} \colon \calE \Gamma \to \calE \Delta$.
	Moreover, $\bar{f}$ maps thick ends to thick ends and thin ends to thin ends.
\end{proposition}

\begin{proof}
	Since $\V\Gamma$ is dense in $\V\Gamma \cup \calE\Gamma$, it is clear that a continuous extension has to be unique.
	
	\emph{Claim:} Let $(\alpha_i)_{i \in \NN}$ be a sequence in $\V\Gamma$ converging to $\xi \in \calE \Gamma$. Then, the sequence $(f(\alpha_i))_{i \in \NN}$ converges to a point in $\calE\Delta$ that is independent of the chosen sequence $(\alpha_i)$.
	
	\emph{Proof:} The argument is essentially the same as in the proof of Lemma \ref{lem:NZnotqi}, we leave it as exercise.
	
	The claim shows that the extension $\bar{f}$ exists and $\bar{f}(\calE\Gamma) \subset \calE\Delta$. Continuity we leave as exercise.
	
	To prove surjectivity, let $(\beta_0,\beta_1,\dots)$ be a ray in $\Delta$.
	There exists $K_5 \geq 0$ such that, for every $i \geq 0$, we have $\alpha_i \in \V\Gamma$ with $d_\Delta(\beta_i,f(\alpha_i)) \leq K_5$. It is easy to see that $(f(\alpha_i))_{i \in \NN}$ and $(\beta)_{i \in \NN}$ converge to the same end.
	
	We now prove injectivity. Let $(\dots,\alpha_{-1},\alpha_0,\alpha_1,\dots)$ be a line defining two different ends in $\calE\Gamma$.
	Assume, by contradiction, that $\bar{f}$ maps them to the same end in $\calE\Delta$.
	Let $\Theta$ be a finite subgraph of $\Gamma$ such that the distance between two connected components of $\Gamma \setminus \Theta$ is at least $(2K_5+1) K_3 + K_4+1$.
	Let $\Theta'$ be the $(K_5+1)$-neighbourhood around the subgraph $f(\Theta)$.
	Let $n \geq 0$ be an integer such that $f(\alpha_i) \in \V(\Delta \setminus \Theta')$ for all $i \geq n$ and all $i \leq -n$.
	By assumption, there exists a path $(\beta_0,\dots,\beta_m)$ from $f(\alpha_{-n})$ to $f(\alpha_n)$ in $\Delta \setminus \Theta'$.
	By almost surjectivity, there exist $\alpha_1',\dots,\alpha_m' \in \V\Gamma$ with $d(f(\alpha_i'),\beta_i) \leq K_5$ for all $i=0,\dots,m$. We can choose $\alpha_0'=\alpha_{-n}$ and $\alpha_m'=\alpha_n$. By construction, $\alpha'_1,\dots,\alpha'_{m-1} \notin \V\Theta$.
	Note that $d(f(\alpha'_i),f(\alpha'_{i+1})) \leq 2K_5+1$ and therefore $d(\alpha'_i,\alpha'_{i+1}) \leq (2K_5+1)K_3 + K_4$.
	However, by choice of $\Theta$, one of $\alpha_1',\dots,\alpha_{m-1}'$ has to be in $\V\Theta$. This is a contradiction.	
	
	A continuous bijection between two compact Hausdorff spaces is a homeomorphism.
	
	The rest is proven via similar ideas.
\end{proof}

Lemma \ref{lem:NZnotqi} is also a direct corollary of Proposition \ref{prop:qi_graphs}.

%

\begin{exercise}
	Let $d,d' \geq 3$ be positive integers. Then, the $d$-regular tree and the $d'$-regular tree are quasi-isometric.
\end{exercise}

\subsection{Willis theory}

In his pioneering paper \cite{w94} Willis laid the foundation for the modern study of tdlc groups.
He introduced concepts that are still central in their structure theory.

\begin{definition}
	Let $G$ be a tdlc group. The \emph{scale function} on $G$ is the function $s_G \colon G \to \NN$ defined by
	\[
		s_G(g) \coloneqq  \min\{[U : U \cap gUg^{-1}] \mid U \leq G \text{ compact, open} \}.
	\]
	Any compact, open subgroup $U \leq G$ achieving this minimum is called \emph{tidy} for $g$.
	The group $G$ is called \emph{uniscalar} if $s_G(g)=1$ for every $g \in G$.
\end{definition}

The original definition of tidy subgroups is more complicated, but it turns out to be equivalent to this one.

\begin{fact}
	We collect a few properties.
	\begin{enumerate}
		\item The function $s_G \colon G \to \NN$ is continuous.
		\item For all $g \in G$ and $n \in \NN$ we have $s_G(g^n)=s_G(g)^n$.
		\item For all $g \in G$ such that $\langle g \rangle$ has compact closure, we have $s_G(g)=1$. This follows from the classical fact that in a tdlc group, every compact subgroup is contained in a compact, open subgroup.
		\item Denote the modular function on $G$ by $\Delta_G$. Then $\Delta_G(g)=s_G(g)/s_G(g^{-1})$ for every $g \in G$.
		\item A compact, open subgroup is tidy for $g$ if and only if it is tidy for $g^{-1}$.
	\end{enumerate}
\end{fact}

Using a construction similar to Cayley--Abels graphs, M\"oller proved the following characterisation of tidy subgroups.

\begin{proposition}[\cite{Moeller2002}, Corollary 3.5]\label{prop:scale_power}
	Let $G$ be a tdlc group and $g \in G$. Then, $U \leq G$ is tidy for $g$ if and only if $[U : U \cap g^n U g^{-n}] = [U : U \cap gUg^{-1}]^n$ for all $n \in \NN$.
\end{proposition}

\section{Constructions and examples}

\subsection{Constructions and quasi-isometry}
\label{sect:constructions}

We give two different constructions to obtain a Cayley--Abels graph $\Gamma$ for a cgtdlc group $G$.
Fix a compact, open subgroup $B \leq G$. In each construction, $B$ will be the stabilizer of a vertex. The vertex set will always be the set of cosets $G/B$. Remembering that by the orbit-stabilizer theorem, for any Cayley--Abels graph $\Delta$ for $G$ and $\alpha \in \V\Delta$
with $G_\alpha =B$ the map $G/B \to \V\Delta, \, gB \mapsto g \alpha$ is a bijection intertwining the action of $G$ on $G/B$ and on $\V\Delta$. So it is not surprising that a construction of a Cayley--Abels graph would just take $G/B$ as vertex set.

\begin{construction} \label{cons:1}
		Let $K$ be a compact generating set for $G$. Let $\hat{\Gamma}$ be the Cayley graph of $G$ with generating set $K$. Note that this graph is connected, but not locally finite. Let $\sim$ denote the coset relation $g \sim h \Leftrightarrow gB=hB$; it is a $G$-congruence on $\hat{\Gamma}$. Now a Cayley--Abels graph is the quotient $\Gamma \coloneqq  \hat{\Gamma}/\sim$. The action of $G$ on $\Gamma$ is induced by the action of $G$ on $\hat{\Gamma}$. In other words, it is via left-multiplication on the coset space $G/B$.
\end{construction}

\begin{construction}\label{cons:2}
		Let $S \subset G$ be a finite subset satisfying the four conditions
	\begin{enumerate}
		\item $S=S^{-1}$
		\item $S \cap B =\emptyset$
		\item $BSB = SB = BS$
		\item $\langle S,B \rangle = \langle S \rangle B = G$.
	\end{enumerate}
	Let $\V\Gamma \coloneqq  G/B$ and $\A\Gamma \coloneqq  \{ (gB,gsB) \mid s \in S, \, g \in G \}$. The action of $G$ on $\Gamma$ is via left-multiplication on the coset space $G/B$.
\end{construction}

%
	
\begin{remark}\label{rem:BSB}
	In Construction \ref{cons:2}, the set of neighbours of the vertex $B$ is the set of cosets $SB$. Thus the equation $BSB = SB$ is interpreted as ``the subgroup $B$ leaves the neighbours of the vertex $B$'' invariant, which is clear from $B=G_B$.
	It is not a difficult exercise to prove that the three sets in Condition 3 are equal to $\{g \in G \mid gB \in \N(B) \}$, the set of group elements sending the vertex $B$ to one of its neighbours.
\end{remark}

\paragraph{Attention!}
	Do not confuse the Cayley--Abels graph with the Schreier graph, usually defined for finitely generated groups. In a Schreier graph, the vertex set is the set of right cosets $B \quot G$ for a subgroup $B \leq G$.
	The arcs are $(Bg,Bgs)$ for $s$ in a finite, symmetric generating set of $G$. There is in general no action of $G$ on its Schreier graph!

	Now we prove that the constructions above indeed all give Cayley--Abels graphs, and that they are actually general.
	
\begin{proposition}
	Let $G$ be a cgtdlc group.
	\begin{enumerate}
		\item In Constructions \ref{cons:1} and \ref{cons:2}, the graph $\Gamma$ is a Cayley--Abels graphs for $G$ and $B$ is the stabilizer of some vertex.
		\item Given any Cayley--Abels graph $\Delta$ for $G$ and any stabilizer $B \leq G$ of a vertex in $\Delta$, there exists a compact subset $K \subset G$, a finite subset $S \leq G$ and a $G$-equivariant
		 graph isomorphism $\phi \colon \Delta\to\Gamma$, where $\Gamma$ is as in Construction \ref{cons:1} and \ref{cons:2}. 
	\end{enumerate}
\end{proposition}

\begin{proof}
	Construction \ref{cons:1}: By Lemma \ref{lem:quotientgraph} we know that $G$ acts vertex-transitively with compact, open vertex stabilizers on $\Gamma$. It is clear that $B$ is the stabilizer of the vertex $B$. Also, $\Gamma$ is connected because $\hat{\Gamma}$ is. It is left to show that $\Gamma$ is locally finite. It is enough to prove that the vertex $B$ has finitely many neighbours. Clearly $B$ is adjacent to $gB$ if and only if there exist $b,b' \in B$ and $k \in K$ with $bk=gb'$, i.e. if and only if $gB \subset BKB$. By compactness of $BKB$ this holds only for finitely many cosets $gB$.
	
	Construction \ref{cons:2}: First we want to show that $\Gamma$ is locally finite. 
	It suffices to prove that the neighbours of a coset $gB$ described in Construction \ref{cons:2} is independent of the coset representative, i.e.
	we prove that if $gB =g'B$ then for every $s \in S$ there exists $s' \in S$ with $gsB = g's'B$. This follows from the third condition. The forth condition implies that $\Gamma$ is connected. It is clear that $B \leq G$ is the stabilizer of the vertex $B \in \V\Gamma$ and that the action of $G$ on the vertices of $\Gamma$ is transitive. To see that this action is by graph automorphisms, note that if $(gB,gsB)$ is an arc in $\Gamma$, then so is $(hgB,hgsB)$ for all $h \in G$.
	
	We now prove Part 2. of the proposition. Let $\Delta$ be a Cayley--Abels graph for $G$, let $\alpha \in \V\Delta$ and denote $B\coloneqq G_\alpha$. Choose a finite, symmetric subset $S \subset G$ such that $S\alpha$ is the set of neighbours of $\alpha$. Define $K\coloneqq SB$. It is an exercise to show that $$K = SB = BSB = BS=\{g \in G \mid g\alpha \in \N(\alpha) \}.$$	
	By the orbit-stabilizer theorem the map $\phi \colon G/B\to\V\Delta , hB \mapsto h\alpha$ is a bijection and it clearly intertwines the action of $G$.
	Now, identifying $G/B$ with $\V\Gamma$ in Construction \ref{cons:1} and \ref{cons:2}, we are left with proving that $\phi$ is a graph isomorphism.
	
	In Construction \ref{cons:1}, two vertices $gB$ and $hB$ are adjacent if and only if there exist $b,b' \in B$ and $k \in K$ with $gb = hb'k$. By definition of $K$ that is equivalent to there existence of an $s \in S$ with $gB = hsB$.
	By definition $\phi(gB)=g\alpha$ and $\phi(hsB)=hs\alpha$.
	Two vertices $g\alpha,\beta \in \V\Delta$ with $g \in G$ are neighbours if and only if $\alpha,g^{-1}\beta$ are are also neighbours. That is the case if and only if there exists an $s \in S$ with $s\alpha=g^{-1}\beta$, i.e. $gs\alpha = \beta$.
	This shows that $\phi$ is indeed a graph isomorphism $\Gamma \to \Delta$ in Construction \ref{cons:1}.
	
	Also in Construction \ref{cons:2}, two vertices $gB$ and $hB$ are adjacent if and only if there exists $s \in S$ with $gB=hsB$. The rest follows as above.
\end{proof}

The following is a generalization of the fact that every group acting freely and transitively on the vertices of a locally finite, connected graph is finitely generated.

\begin{proposition}[\cite{KronMoller2008} Theorem 2.2]
	Let $G$ be a tdlc group admitting a Cayley--Abels graph. Then, $G$ is compactly generated.
\end{proposition}

\begin{proof}
	Let $\Gamma$ be a Cayley--Abels graph for $G$ and $\alpha \in \V\Gamma$.
	Let $S \subset G$ be finite such that $S\alpha$ consist of all neighbours of $\alpha$. We show that $G_\alpha \cup S$ is a compact generating set for $G$. Compactness is clear.
	
	 By Lemma \ref{lem:vertextrans} the finitely generated group $\langle S \rangle$ acts vertex-transitively on $\Gamma$. Let $g \in G$ be arbitrary. Choose an element $h \in \langle S \rangle$ with $g\alpha = h\alpha$. Then, $h^{-1}g \in G_\alpha$, hence $g \in \langle G_\alpha \cup S \rangle$ and we are done.
\end{proof}

\begin{remark}
	If $G$ is finitely generated and $B=\{1\}$, the Cayley--Abels graph is simply the Cayley graph. A Cayley graph comes with a legal, regular coloring preserved by the action of $G$. For a general Cayley--Abels graph, one has such a coloring if and only if $B$ is a normal subgroup of $G$, which is the case if and only if $B$ equals the kernel of the action of $G$ on $\Gamma$.
\end{remark}

The next proposition says that two Cayley--Abels graphs for a cgtdlc group $G$ are quasi-isometric and the quasi-isometry behaves well with the action of $G$ on the vertices.

\begin{proposition} [\cite{KronMoller2008} Theorem 2.7, Theorem 2.7+, Theorem 3.14]
	\label{prop:CAqi}
  Let $G$ be a cgtdlc group and let $\Gamma$ and $\Delta$ be Cayley--Abels graphs for $G$.
  Then, $\Gamma$ and $\Delta$ are quasi-isometric.
  
  More precisely, there exists a quasi-isometry $\phi \colon \Gamma \to \Delta$ that
  is \emph{quasi-$G$-equivariant}, i.e. there exists a constant $K >0$ such that for every $\alpha \in \V\Gamma$ and every $g \in G$ we have $d_\Delta(g\phi(v),\phi(gv)) < K$.
  
  Each $\phi$ satisfying this properties
   extends uniquely to a continuous map $\overline{\phi} \colon \V \Gamma \cup \calE \Gamma \to \V\Delta \cup \calE\Delta$ restricting to a $G$-equivariant
    homeomorphism $\calE \Gamma \to \calE \Delta$. 
\end{proposition}

\begin{proof}
	Let $\gamma \in \V\Gamma$ and $\delta \in \V\Delta$. It is enough to prove the two cases $G_\gamma = G_\delta$ and $G_\gamma \leq G_\delta$, the rest follows from transitivity considering $G_\gamma \cap G_\delta$.
	Let $d_\Gamma$ be the distance on $\Gamma$ and $d_\Delta$ the distance on $\Delta$.
	
	We first consider the first case. In that case, we can identify $\V\Gamma = \V\Delta = G/G_\gamma$. We want to show that the identity map between $\V\Gamma$ and $\V\Delta$ is a quasi-isometry. The second condition (almost surjectivity) is trivially satisfied because the identity is surjective.
	For the inequality, let $K_1 > 0$ satisfy $\N_\Gamma(\gamma) \subset \V\B_\Delta(\delta,K_1)$, where $\N_\Gamma(\gamma)$ is the set of neighbours of $\gamma$ in $\Gamma$ and $\B_\Delta(\delta,K_1)$ is the subgraph of $\Delta$ that is the ball of radius $K_1$ around $\delta$. Let $(e_1,\dots,e_n)$ be a geodesic path in $\Gamma$ between two vertices $\alpha \coloneqq  o(e_1)$ and $\beta \coloneqq  t(e_n)$ with $d_\Gamma(\alpha,\beta)=n$.
	Since $G$ acts vertex-transitively by isometries on $\Gamma$ and $\Delta$, we know that  $\N_\Gamma(o(e_i)) \subset \B_\Delta(o(e_i),K_1)$ for all $i=1,\dots,n$.
	In particular $d_\Delta(o(e_i),t(e_i)) \leq K_1$.
	Therefore $d_\Delta(\alpha,\beta) \leq d_\Delta(o(e_1),t(e_1)) + \dots + d_\Delta(o(e_n),t(e_n)) \leq n K_1$.
	
    The other inequality follows by symmetry, exchanging the roles of $\Gamma$ and $\Delta$.
    
	We now consider the second case. We identify $\V\Gamma = G/G_\gamma$.
	The relation given by $G_\delta$-cosets, i.e. $g G_\gamma \sim h G_\gamma$ if and only if $g G_\delta = h G_\delta$, is a $G$-congruence with finite equivalent classes on $\Gamma$. By Lemma \ref{lem:quotientgraph} the group $G$ acts continuously on the quotient graph $\Gamma/\sim$. We can identify $\V(\Gamma/\sim) = G/G_\delta$.
	It is easy to see that $\Gamma/\sim$ is connected, locally finite and the stabilizer of the vertex $G_\delta \in \V(\Gamma/\sim)$ is $G_\delta \leq G$, so $\V\Gamma/\sim$ is a Cayley--Abels graph for $G$.
	By the first case, we assume without loss of generality that $\Delta = \Gamma/\sim$.
		
	We want to prove that the quotient map $\phi \colon \V\Gamma \to \V\Delta,\, gG_\gamma \mapsto gG_\delta$ is a quasi-isometry. It is clearly surjective.
	One of the inequalities is obvious. If $g,h \in G$ then $d_\Delta(gG_\delta, hG_\delta) \leq d_\Gamma(gG_\gamma,hG_\gamma)$.
	
	The pre-image $\phi^{-1}(\N(\delta) \cup \{\delta\})$ is finite, its cardinality is $(|\N(\delta)|+1)[G_\delta:G_\gamma]$. Let $K$ be its diameter.
	This implies that for every $\delta' \in \N(\delta)$ and every $\alpha,\beta \in \V\Gamma$ with $\phi(\alpha)=\delta$ and $\phi(\beta)=\delta'$ we have $d_\Gamma(\alpha,\beta) \leq K$.
	For every vertex $\alpha \in \V\Gamma$ and every $g \in G$ we have $\phi(g\alpha)=g\phi(\alpha)$. Consequently, for every $\alpha' \in \V\Delta$ the diameter of $\phi^{-1}(\N(\alpha'))$ is $K$.
	Let $\alpha,\beta \in \V\Gamma$ be two vertices, set 
	 $n \coloneqq  d_\Delta(\phi(\alpha),\phi(\beta))$ and let $(e_1,\dots,e_n)$ be a geodesic path in $\Delta$ with $o(e_1)=\phi(\alpha)$ and $t(e_n)=\phi(\beta)$.
	Similar to the first case, we can see that $d_\Gamma(\alpha,\beta) \leq Kn$, so setting $K_1 \coloneqq  1/K$ we are done.
	
	Existence of $\overline{\phi}$ was proven in Proposition \ref{prop:qi_graphs}.
	Equivariance is a simple consequence of quasi-equivariance of $\phi$.
\end{proof}

As a corollary, we can split up elements of a $G$ into three different types as in Proposition \ref{prop:types_of_elements}.

\begin{definition} \label{def:group_types_of_elements}
	Let $G$ be a cgtdlc group and let $\Gamma$ be a Cayley--Abels graph for $G$.
	An element $g \in G$ is called \emph{elliptic, hyperbolic} or \emph{parabolic} according to whether it acts like an elliptic, hyperbolic or parabolic element on $\Gamma$.
\end{definition}

\begin{corollary}[\cite{KronMoller2008}, Theorem 3.32]
	\label{cor:group_types_of_elements}
	Let $G$ be a cgtdlc group and let $g \in G$. Whether $g$ is elliptic, hyperbolic or parabolic as in Definition~\ref{def:group_types_of_elements} is independent of the choice of $\Gamma$.
\end{corollary}

\subsection{Normal subgroups}

It should not come as a surprise that Cayley--Abels graphs behave well with respect to quotients.
%

\begin{proposition}\label{prop:quotient_CA_by_normal}
	Let $G$ be a cgtdlc group and $N \triangleleft G$ a closed, normal subgroup.
	Let $\Gamma$ be a Cayley--Abels graph of $G$.
	\begin{enumerate}
		\item The quotient graph $N \quot\Gamma$ is a Cayley--Abels graph for the quotient group $G/N$.
		\item We have $\val(N \quot\Gamma) \leq \val(\Gamma)$.
		Equality holds if and only if for one, and hence every, vertex $\alpha \in \V\Gamma$ the compact set $\{n \in N \mid n\alpha \in \B(\alpha,2)\}$ is contained in the kernel of the action of $G$ on $\Gamma$.
	\end{enumerate}
\end{proposition}

\begin{proof}
	This is a consequence of Lemma \ref{lem:quotient_by_normal}.
\end{proof}

\begin{exercise} Recall from Exercise \ref{exer:local_action_normal_subgp} that even though normal subgroups do not act vertex-transitively in general, we can still talk about their ``local action".
	Show that the condition in 2. implies in particular that the local action of $N$ on $\Gamma$ is trivial.
\end{exercise}

\subsection{Examples}

\begin{example}
	Let $G$ be finitely generated and $B=\{1\}$. Then $G$ is discrete and every Cayley--Abels graph is a Cayley graph.
\end{example}

\begin{exercise}
	Let $G$ be a tdlc group. The graph only consisting of one vertex is a Cayley--Abels graph for $G$ if and only if $G$ is compact.
\end{exercise}

\begin{example}
	Let $\T$ be a regular tree and $G \leq \AutT$ any closed, vertex-transitive subgroup (for example a Burger--Mozes universal group $\U(F)$ or the stabilizer of a chosen end). Then $\T$ is a Cayley--Abels graph for $\T$.
\end{example}

\begin{proposition}[\cite{BMW12} Prop. 14]
	Let $G$ act co-compactly on a locally finite, connected cell complex $X$.
	Then, the $1$-skeleton of $X$ is quasi-isometric to a Cayley--Abels graph of $G$. If $G$ acts vertex-transitively on $X$, the $1$-skeleton of $X$ is itself a Cayley--Abels graph for $G$.
\end{proposition}

The next example allows us to get new Cayley--Abels graphs out of old ones.
If a finitely generated group $G=\langle S \rangle$, where $S$ is a finite, symmetric generating set, acts on a set $X$, the \emph{action graph} $\Gamma$ is defined as follows.
The vertex set is $X$ and the arc set is $\{(x,sx) \mid s \in S \}$.

\begin{example}\label{expl:newCayofOld}
	Let $\Gamma$ be a Cayley--Abels graph for $G$ and assume that the action is locally transitive. Then, the following is another Cayley--Abels graph for $G$.
	The vertices of $\Gamma'$ are the arcs of $\Gamma$.
	For all $\beta \in \V\Gamma$ and distinct $\alpha,\alpha' \in \N(\beta)$ there is an edge between $(\alpha,\beta)$ and $(\alpha',\beta)$. There also is an edge between the arcs $(\alpha,\beta)$ and $(\beta,\alpha)$.	
\end{example}

Less formally, this means that every vertex of $\Gamma$ is replaced with a complete graph.

\begin{exercise}
	Modify the construction in Example \ref{expl:newCayofOld} such that you get a graph of smaller degree if the local action is not, for example, primitive.
	
	What could you get in the case where $\Gamma$ is a regular tree of degree $d$ and $G=\U(D_{d})$ is a Burger--Mozes universal group, where the dihedral group $D_d$ is the symmetry group of the regular $d$-gon?
\end{exercise}

\begin{example}\label{expl:arc_graph}
	A similar construction as Example \ref{expl:newCayofOld} that works also if $G$ is only edge-transitive is the following.
	After doing the construction, we contract all the edges of the form $\{(\alpha,\beta),(\beta,\alpha)\}$ to a vertex.
	
	More explicitly, let $\Gamma$ be a connected, locally finite graph on which $G$ acts edge-transitively. Then, the following is a Cayley--Abels graph for $G$.
	The vertices of $\Gamma'$ are the edges of $\Gamma$. There is an edge between $\{\alpha,\beta\}$ and $\{\alpha',\beta\}$ if there exists $g \in G_\beta$ with $\{g\alpha,g\beta\} = \{\alpha',\beta\}$.
\end{example}

\begin{exercise}
	Do this for the automorphism group of a biregular tree.
\end{exercise}

The author does not know, but would be interested in, the answer to the following question (it is true for $d'=2$ or if $d,d'$ are small).

\begin{problem}
	Let $\T_{d,d'}$ be the biregular tree with valencies $d \geq 3$ and $d' \geq 2$.
	Does the construction from Example \ref{expl:arc_graph} give a Cayley--Abels graph of minimal possible valency for $\Aut(\T_{d,d'})$?
\end{problem}

Few examples of Cayley--Abels graphs have been explicitly drawn.
It is usually just a tool in a proof, like for us in Theorem \ref{thm:esschiefseries}.

\begin{problem}
	Give a good explicit description of a Cayley--Abels graph for Neretin's group.
\end{problem}

\subsection{Tree automorphism groups}

In this subsection we want to get a better understanding of $\Aut(\T_d)$ for a $d$-regular tree $\T_d$. 
It should not go without mentioning that $\Aut(\T_d)$ is a tremendously rich group that can probably be rightfully regarded as the most fundamental example in the theory of tdlc groups.
We start with the special case of Proposition \ref{prop:types_of_elements}.

\begin{proposition}\cite{Tits1970}
	Let $\varphi \in \Aut(\T_d)$. Then, $\varphi$ either fixes a vertex, transposes an edge or translates along an infinite line.
\end{proposition}

Note that, in particular, there are not parabolic elements - which is clear, because there are no thick ends.

Next we look at compact subgroups.

\begin{proposition}\label{prop:cpct in AutT}
	Let $K \leq \Aut(\T_d)$ be a compact subgroup.
	Then $K$ fixes a vertex or stabilizes an edge.
	In particular, vertex stabilizers and edge stabilizers are maximal among compact subgroups and the only compact, normal subgroup of $\Aut(\T_d)$ is the trivial group.
\end{proposition}

In general, a tdlc group might not have maximal compact subgroups; not even if $\T_d$ is a Cayley--Abels graph. Let for example $\xi \in \calE \T_d$ be an end and $G=\Aut(\T_d)_\xi$ its stabilizer. Let $(\alpha_0,\alpha_1,\dots)$ be any ray representing $\xi$.
Then $G_{\alpha_0} \leq G_{\alpha_1} \leq \dots$ is an ascending sequence of compact, open subgroups of $G$ that never stabilizes. In the next lemma we want to understand the relationship between edge and vertex stabilizers.

Note that in $\Aut(\T_d)$ not just all vertex stabilizers are conjugate, but also all edge stabilizers because the local action is transitive.

\begin{lemma}\label{lem:rel betw max cpcts in AutT}
	Let $\alpha,\beta$ be neighboring vertices in $\T_d$. 
	Let $U := G_{\alpha}$ and $V := G_{\{\alpha,\beta\}}$.
	Then
	\begin{enumerate}
		\item $U \cap V = G_{(\alpha,\beta)}$
		\item $[V : U \cap V] = 2$
		\item $[U : U \cap V] = d$
		\item The group $U \cap V$ is a maximal subgroup of both $U$ and $V$.
		\item The only compact subgroups containing $U \cap V$ are $U$, $G_\beta$ and $V$.
		\item $N_{\AutTd}(U) = U$
		\item $N_{\AutTd}(V) = V$
	\end{enumerate}
\end{lemma}

\begin{proof}
	Exercise.
\end{proof}

Let $\AutTd^+ \leq \Aut(\T_d)$ be the subgroup generated by all vertex stabilizers.
If we consider the natural bipartition $\V\T_d = V_1 \sqcup V_2$ on the vertices (i.e. such that vertices of $V_1$ are only adjacent to vertices of $V_2$ and vice versa), then 
it is not hard to show that $\Aut(\T_d)^+ = \Aut(\T_d)_{\{V_1\}}$ and $[\Aut(\T_d) : \Aut(\T_d)^+] = 2$.

\begin{theorem}[\cite{Tits1970}]
	The group $\AutTd^+$ is simple. It is the unique non-trivial, proper, normal subgroup of $\Aut(\T_d)$.
\end{theorem}

We stress the fact that all proper, open subgroups of $\Aut(\T_d)^+$ are compact and as a consequence, all compact, open subgroups in $\Aut(\T_d)$ have compact normalizers.

The following proposition about tree automorphism groups is well-known to experts, but its proof seems to be rarely presented. It relies heavily on maximal compact subgroups and their relationship with each other.

\begin{proposition}
	Let $d \neq d'$ be non-negative integers.
	Then, $\T_{d'}$ cannot be a Cayley--Abels graph for $\AutTd$, except in the case $d'=0$ and $d=1$.
\end{proposition}

\begin{proof}
	We assume $d,d' \geq 2$, the other cases are clear.
	
	Assume that $G \coloneqq  \AutTd$ acts vertex-transitively with compact, open vertex stabilizers on $\T_{d'}$. We have to prove that $d=d'$. Recall that $\AutTd$ does not have any compact, normal subgroups other than the trivial subgroup, so the action is faithful and we can say $G \leq \Aut(\T_{d'})$. Let $U \leq G$ be a vertex stabilizer and $V \leq G$ the stabilizer of an adjacent edge in $\T_d$.
	Recall that $U$ and $V$ are maximal among compact subgroups.
	
	\emph{Claim:} Both $V$ and $U$ are equal to the stabilizer in $G$ of a vertex or an edge in $\T_{d'}$, but not an arc.
	
	The first part follows directly from Proposition \ref{prop:cpct in AutT}.
	We show the second part for $U$; the proof for $V$ is verbatim.
	Assume by contradiction that $U \leq G_{(\alpha,\beta)}$ for some arc $(\alpha,\beta)$. Maximality implies $U = G_\alpha = G_\beta$. There is $g \in G$ with $g \alpha = \beta$. But then $U = G_\beta = gG_\alpha g^{-1}$ and $g \in N_G(U)$. Lemma \ref{lem:rel betw max cpcts in AutT} implies $g \in U$, contradiction.
	This finished the proof of the claim.
	
%
%
%
	
	Since $U$ and $V$ are not conjugate, they cannot both fix a vertex in $\T_{d'}$.
	This leaves three cases.
	
	\emph{Case 1:} Both $U$ and $V$ stabilize an edge.
	
	Let $V = G_{\{\alpha,\beta\}}$ and $U = G_{\{\gamma,\delta\}}$; assume that $\alpha$ is closer to $\gamma$ and $\delta$ than $\beta$.
	The group $G_{\alpha}$ clearly satisfies $V \cap G_{\alpha} \geq V \cap U$ and in particular $G_\alpha \geq V \cap U$ and therefore, since $U$ does not fix a vertex by assumption, we get $G_\alpha = U \cap V$ by Lemma \ref{lem:rel betw max cpcts in AutT}.
	Then $G_\alpha$ fixes both $(\alpha,\beta)$ and $(\gamma,\delta)$.
	Maximality gives $G_\alpha = G_\beta = G_\gamma = G_\delta$.
	Propagating the argument we see that there are infintely many vertices $\alpha_1,\alpha_2,\dots$ with $G_\alpha = G_{\alpha_i}$. But every $g_i \in G$ with $g_i \alpha = \alpha_i$ is contained in the normalizer of $G_\alpha$; contradiction to the fact that $G_\alpha$ must have a compact normalizer.	
	
	\emph{Case 2:} The group $U$ stabilizes an edge and the group $V$ stabilizes a vertex.
	
	Let $\alpha,\beta \in \V\T_{d'}$ such that $U=G_{\{\alpha,\beta\}}$.
	Then $U \cap V$ is contained in $G_\beta$ and $G_\gamma$, which are two different conjugates of $V$. This is a contradiction to Lemma~\ref{lem:rel betw max cpcts in AutT}.
	
	\emph{Case 3:} The group $U$ stabilizes a vertex and the group $V$ stabilizes an edge.
	
	Since, in $\T_d$ and in $\T_{d'}$, the subgroup $U$ stabilizes a unique vertex, the map
	$\alpha \mapsto G_\alpha$ defines a bijection $\V \T_{d'} \to \{gUg^{-1} \mid g \in G\}$
	and $\V \T_{d} \to \{gUg^{-1} \mid g \in G\}$.
	In particular, we get a bijection $\phi \colon \V\T_d \to \V\T_{d'}$ satisfying $G_\alpha = G_{\phi(\alpha)}$ for all $\alpha \in \V\T_d$.
	To show that this bijection is a tree isomorphism, it is enough to show that it is a graph morphism.
	Now the key observation is that neighbors can be detected via inclusion.
	Namely, 
	if $\{\alpha,\beta\}$ is an edge in $\T_d$ then there is no $\gamma \in \V\T_d$ with $G_\gamma \leq G_\alpha \cap G_\beta$. But then there is no $\gamma' \in \V\T_{d'}$
	with $G_{\gamma'} \leq G_{\phi(\alpha)} \cap G_{\phi(\beta)}$ and thus $\{\phi(\alpha),\phi(\beta)\}$ is an edge in $\T_{d'}$.	
	This implies that $\phi$ is a graph morphism and the proof is complete..
%
%
%
%
%
\end{proof}

\section{Lifting to trees}\label{sect:trees}

In this section, we take a slightly different look at graphs: We consider them as topological realizations of one-dimensional cell complexes. In particular, they are metric spaces, and every edge is isometric to the unit interval. The automorphism group of a graph is then the same as its isometry group.

\begin{definition}
	Let $X,Y$ be metric spaces. A \emph{local isometry} is a continuous, open map $f \colon X \to Y$ such that every $x \in X$ has a neighbourhood $B \subset X$ such that $f|_B \colon B \to f(B)$ is an isometry.
\end{definition}

With this definition, a graph morphism $f \colon \Gamma \to \Delta$ is a local isometry if and only if for every vertex $\alpha \in \V\Gamma$, the restriction $f|_{\N(\alpha)}$ is a bijection $\N(\alpha) \to \N(f(\alpha))$. In that case, if $\Gamma$ and $\Delta$ are in addition connected, $f$ is already a covering map. Clearly $\Gamma$ is a tree if and only if it is simply connected and then a local isometry is automatically a tree automorphism.

We will need the well-known theorem about lifting maps to the universal covering.

\begin{theorem}
	Let $(X,x_0)$ and $(Y,y_0)$ be path connected, locally path connected metric spaces with a base point.
	Let $\pi \colon (\tilde X, \tilde x_0) \to (X,x_0)$ be a universal covering such that $\pi$ is a local isometry.
	Let $f \colon (Y,y_0) \to (X,x_0)$ be a continuous local isometry.
	
	Then, there exists a unique local isometry $\tilde f \colon (Y,y_0) \to (\tilde X, \tilde x_0)$ with $\pi \circ \tilde{f} = f $ if and only if $f_*(\pi_1(Y,y_0)) = \{1\}$.
	
	If $f$ is a covering map, then $\tilde f$ is a covering map.
		
	\begin{equation*}
	\xymatrix{
		 & (\tilde X, \tilde x_0) \ar[d]^{p} \\
		(Y,y_0) \ar[r]^{f}\ar@{-->}[ur]^{\tilde{f}}  & (X,x_0)
	}
	\end{equation*}
\end{theorem}

Noting that the universal covering of a graph is a tree, we get the following corollary.

\begin{corollary}\label{cor:universalcover}
	Let $\Gamma$ be a locally finite graph, $\alpha \in \V\Gamma$ and $g \in \Aut(\Gamma)$.
	Let $\T$ be a tree, $\tilde \alpha \in \V\T$ and $\pi \colon \T \to\Gamma$ be a universal covering map with $\pi(\tilde \alpha)=\alpha$.
	Choose $\tilde \beta \in \V\T$ with $\pi(\tilde{\beta})= g^{-1}(\alpha)$.
	
	Then, there exists a unique tree automorphism $\overline{g} = \widetilde{g \circ \pi} \in \Aut(\T)$ with $\overline{g}(\tilde \beta)=\tilde \alpha$ such that $ \pi \circ \overline{g}=g \circ \pi $.
	
	\begin{equation*}
	   \xymatrix{
	   (\T,\tilde{\beta}) \ar@{-->}[r]^{\overline{g}} \ar[d]^{\pi} & (\T,\tilde{\alpha}) \ar[d]^{\pi} \\
	   (\Gamma,g^{-1}(\alpha)) \ar[r]^{g} & (\Gamma,\alpha)
		}
	\end{equation*}
\end{corollary}

In the situation of the above commutative diagram, we say that $\overline{g}$ \emph{covers} $g$.
We see that it is possible to construct a tree automorphism covering any given graph automorphism, but it is important to emphasize that choices of base points were involved.
As a consequence, there is, in general, no group homomorphism of the form $G \to  \widetilde{G},\, g \to \overline{g}$, as is easily seen in the following example.

\begin{example}
	Let $\Gamma$ be a triangle and $G = \mathbb{Z}/3\mathbb{Z}$.
	Then $\T$ is the bi-infinite line and $\AutT$ does not have any elements of order $3$.
	Thus there cannot be a group homomorphism of the form $G \to \AutT,\, g \mapsto \overline{g}$.
\end{example}

\begin{definition}\label{def:tildeG}
	Let $(\Gamma,\alpha)$ be a graph and let $\pi \colon  (\T, \tilde{\alpha}) \to (\Gamma,\alpha)$ be a universal covering of $\Gamma$. Let $G \leq \Aut(\Gamma)$.
	Then
	\[
	\widetilde{G} \coloneqq  \{ \overline{g} \in \Aut(\T) \mid \exists g \in G \colon  \pi \circ \overline{g} = g \circ \pi \}.
	\]
\end{definition}

\begin{lemma}
	With the notation from Definition \ref{def:tildeG}, the group $\widetilde{G}$ is a subgroup of $\AutT$. The map $\Phi \colon \widetilde{G} \to G, \, \overline{g} \mapsto g$ is a well-defined group homomorphism and its kernel is the set of deck transformations.
%
\end{lemma}

\begin{proof}
	Clearly the identity map covers the identity map. Also, for every $\overline{g} \in \widetilde{G}$, if $g \circ \pi = \pi \circ \overline{g}$ then $g^{-1} \circ \pi = \pi \circ \overline{g}^{-1}$, so $\overline{g}^{-1}$ covers $g^{-1}$ and $\overline{g}^{-1} \in \widetilde{G}$.
	Let $\overline{g}, \overline{h} \in \widetilde{G}$.
	There exist $g,h \in G$ such that $\pi \circ g = \overline{g} \circ \pi$ and $\pi \circ h = \overline{h} \circ \pi$. Then $\pi \circ h \circ g = \overline{h} \circ \pi \circ g = \overline{h} \circ \overline{g} \circ \pi$.
%
%
	In particular $\overline{h} \circ \overline{g}$ covers $hg$, therefore $\overline{h} \circ \overline{g} \in \widetilde{G}$. 
	
	Next we show that the map $\Phi$ is well-defined. Assume that there exist $g,g' \in G$ and $\overline{g} \in \widetilde{G}$ with $\pi \overline{g} = g \pi = g' \pi$.
	It is surjectivity of $\pi$ that implies $g = g'$.
	
	We showed above that $\overline{h} \circ \overline{g}$ covers $hg$. This proves that $\Phi$ is a group homomorphism. The kernel of $\Phi$ is the set of all $\overline{1} \in \widetilde{G}$ such that $\pi \circ \overline{1} = \pi$. This is by definition the set of deck transformations.
\end{proof}

Recall the following fact about deck transformations.

\begin{fact}
	Let $(X,x_0)$ be a path connected, locally path connected topological space and $\pi \colon (\tilde{X},\tilde{x_0}) \to (X,x_0)$ its universal covering. Then, the group of deck transformations is isomorphic to the fundamental group $\pi_1(X,x_0)$.
\end{fact}

We apply this to automorphism groups of graphs. This gives a short exact sequence
\[
\{1\} \to \pi_{1}(\Gamma,\alpha) \to \widetilde{G} \to G \to \{1\}.
\]
Recall that the fundamental group of a locally finite graph is isomorphic to a free group on finitely or countably many generators. Also recall that the group of deck transformations acts freely on the points of the covering space.

\begin{lemma}
	Let $(\Gamma,\alpha)$ be a locally finite graph and $\pi \colon (\T,\tilde{\alpha}) \to (\Gamma,\alpha)$ its universal covering. Assume that $G \leq \Aut(\Gamma)$ is vertex-transitive.
	\begin{enumerate}
		\item The subgroup $\widetilde{G} \leq \Aut(\T)$ is vertex-transitive.
		\item The subgroup $\widetilde{G} \leq \Aut(\T)$ has the same local action as $G \leq \Aut(\Gamma)$.
	\end{enumerate}
\end{lemma}

\begin{proof}
	We first prove 1. Let $\tilde{\beta} \in \V\T$. By assumption there exists $g \in G$ with $g(\alpha)=\pi(\tilde{\beta})$. By Corollary \ref{cor:universalcover} there exists $\overline{g} \in \widetilde{G}$ such that $\overline{g}(\tilde{\beta})=\tilde{\alpha}$.
	
	Now we prove 2. Since we are assuming that $\Gamma$ is simple, the covering map $\pi$ restricts to a bijection $\pi \colon {\tilde{\alpha}} \cup \N(\tilde{\alpha}) \to {\alpha} \cup \N(\alpha)$.
	Clearly $\Phi(\widetilde{G}_{\tilde{\alpha}}) = G_\alpha$.
	Since the set of deck transformations acts freely on the vertices, $\ker(\Phi) \cap \widetilde{G}_{\tilde{\alpha}} = \{1\}$ and $\Phi$ even induces an isomorphism $\widetilde{G}_{\tilde{\alpha}} \to G_\alpha$ and the actions of this group on $\N(\alpha)$ and $\N(\tilde{\alpha})$, respectively, are conjugate via $\pi$. This shows that the local actions are the same.
\end{proof}

Conclusively, we get the following theorem.

	\begin{theorem}
	Let $G$ be a cgtdlc group with Cayley--Abels graph $\Gamma$.
	Let $K$ be the kernel of the action of $G$ on $\Gamma$.
	Let $\T$ be a regular tree with the same valency as $\Gamma$.
	Then, there exists a vertex-transitive, closed subgroup $H \leq \Aut(\T)$ and an embedding $\pi_1(\Gamma) \hookrightarrow H$ such that
	\begin{enumerate}
		\item $G/K \cong H/\pi_1(\Gamma)$,
		\item $G/K$ and $H$ are locally isomorphic, and
		\item $G/K$ and $H$ have the same local actions on $\Gamma$ respectively $\T$.
	\end{enumerate}
\end{theorem}

\section{Similarities to Cayley graphs}
\label{sect:fg}

We will not give many proofs in this section, mostly because they are too long or resemble the finitely generated case too much.
Most of this material can be found, in the more general setting of locally compact groups, in the (highly recommended) book by Cornulier--de la Harpe \cite{CornulierDeLaHarpe2016}.

\subsection{Compact presentation}

Here we closely follow Sections 2.3-2.5 in \cite{dlSTessera19}.

Compact presentability is a non-discrete analogue of finite presentability for discrete groups. Recall that a \emph{presentation} for a group $G$ is a set $S$ together with a subset $R \subset F_S$ such that $G$ is isomorphic to $F_S/\lpres R \rpres$,
where $F_S$ is the free group with basis $S$ and $\lpres R \rpres$ is the smallest normal subgroup of $F_S$ containing $R$.
We write $G = \langle S| R \rangle$ and call $S$ the \emph{generators} and $R$ the \emph{relators}.

\begin{definition}
	A locally compact group $G$ is called \emph{compactly presented} if
	$G$ admits a presentation $G = \langle K|R \rangle$ such that $K \subset G$ is compact and $R$, viewed as subset of the vertex set of the Cayley graph of $F_K$ with generating set $K$, has finite diameter.
\end{definition}

Note that the Cayley graph mentioned in above definition is, in general, not locally finite. Another way of expressing the condition on $R$ is to view its elements as words with letters in $K$ and say that there exists an upper bound on the lengths of these words.

In view of the constructions of a Cayley--Abels graph given in Section \ref{sect:constructions} we would like to ask that a compact generating set contains a given compact, open subgroup.
%

\begin{construction}\label{cons:CayAb2cplx}
	Let $G$ be a cgtdlc group. Let $B \leq G$ be a compact, open subgroup and $K \subset G$ a compact generating set such that $K=BKB$. For example, take $K=SB$ with $S$ as in Construction \ref{cons:2}.
	Let $\phi \colon F_K \to G$ be the obvious homomorphism and let $R \subset \ker(\phi)$. The group homomorphism $\phi$ also defines a graph morphism from the Cayley graph of $F_K$ with generating set $K$, which is a tree, to the Cayley graph $\hat \Gamma$ of $G$ with generating set $K$. We can use the symbol $\phi$ also to denote this graph morphism.
	We use Construction \ref{cons:1} from above.
	Recall that there is an obvious surjective graph homomorphism $\psi \colon \hat \Gamma \to \Gamma$.
	We construct a polygonal $2$-complex $X$ as follows.
	Note that, for every $r \in R$, the graph morphism $\psi \circ \phi$ maps the unique arc from $1$ to $r$ in the Cayley graph of $F_K$ to a loop in $\Gamma$.
	The same holds true for the unique arc from $g$ to $gr$ for any $g \in F_K$.
	The $2$-complex $X$ 
	has $1$-skeleton $\Gamma$ and is now obtained by gluing a polygon along each of these loops.
\end{construction}

The following theorem works just as in the finitely generated case.

\begin{theorem}\label{thm:CayAb2cplx}
	In Construction \ref{cons:CayAb2cplx}, the set
	$R \subset F_K$ generates the kernel of the surjective group homomorphism $\phi \colon F_K \to G$ if and only if the complex $X$ is simply connected.
\end{theorem}

This motivates the definition of a Cayley--Abels $2$-complex.

\begin{definition}
	In the situation of Theorem \ref{thm:CayAb2cplx}, we call $X$ a \emph{Cayley--Abels $2$-complex} of $G$.
\end{definition}

We get the following criterion for a tdlc group to be compactly presented.

\begin{corollary}
	Let $G$ be a tdlc group. It is compactly presented if and only if there exists a Cayley--Abels graph $\Gamma$ and $k \geq 0$ such that the $2$-complex obtained from $\Gamma$ by gluing in polygons along all loops of length at most $k$ is simply connected.
\end{corollary}

%

\subsection{Hyperbolicity}

Recall that a geodesic metric space is called \emph{hyperbolic} if there exists a $\delta >0$ such that for all geodesic triangles the $\delta$-neighbourhood of two sides contains the third. This definition is due to Gromov. Being hyperbolic is, for geodesic metric spaces, invariant under quasi-isomorphisms, so the following is well-defined.

\begin{definition}
	Let $G$ be a cgtdlc group. It is called \emph{hyperbolic} if a Cayley--Abels graph of $G$ with the usual metric is a hyperbolic metric space.
\end{definition}

Useful facts and interesting results about hyperbolic tdlc groups that can be found in \cite{Bywaters19}.

Recall from Corollary \ref{cor:group_types_of_elements} that elements of a cgtdlc group come in three different types: elliptic, hyperbolic and parabolic.

\begin{proposition}[\cite{BMW12}, Theorem 22]
	A hyperbolic cgtdlc group does not have any parabolic elements.
\end{proposition}

For a finitely generated group, it is well-known that a finitely generated, hyperbolic group does not contain a discrete copy of $\ZZ^2$.
We turn now to a statement that is a bit similar in spirit.
The \emph{flat rank} of a tdlc group is an invariant analogous to the rank of a semisimple algebraic group over a local field \cite{BRW07}. Unfortunately, the definition looks very unmotivated at first sight.
The attentive reader might notice that few non-trivial statements are hidden in the definition.

\begin{definition}
	The \emph{flat rank} of $G$ is the supremum over all the ranks of free abelian groups $H/N_H(U)$, where $H$ ranges over all subgroups of $G$ admitting a compact, open subgroup $U \leq G$ that is tidy for every element of $H$ and $N_H(U)$ is the normalizer of $U$ in $H$.
%
%
\end{definition}

The hidden statements are that for such an $H$, the quotient $H/N_H(U)$ of $H$ by the normalizer of $U$ inside $H$ is a free abelian group and its rank is independent of $U$.

If $G$ is a semisimple algebraic group, then its flat rank indeed coincides with its rank.

\begin{theorem}[\cite{BMW12}, Theorem 1]
	Let $G$ be a hyperbolic cgtdlc group. The flat rank of $G$ is at most $1$.
\end{theorem}

Hyperbolic spaces typically come with several types of boundaries, perhaps most commonly with the \emph{Gromov boundary}, the elements of which are equivalence classes of \emph{geodesics}. Two geodesics are equivalent if they stay at bounded distance from one another. For a tree, the Gromov boundary and the set of ends is the same.
Just as with the space of ends, also the Gromov boundary comes with a topology and a quasi-isometry between hyperbolic spaces induces a homeomorphism between the Gromov boundaries.

Tesselations of the hyperbolic plane give examples of hyperbolic graphs where the notions are very different. Those graphs are $1$-ended, but the Gromov boundary is homeomorphic to the circle. Note that their automorphism group is discrete.

\begin{theorem}[\cite{CCMT15} Corollary C]
	If $\Gamma$ is a hyperbolic Cayley--Abels graph for a tdlc group $G$ and the stabilizer of a point in the Gromov boundary acts vertex-transitively, then $\Gamma$ is quasi-isometric to a regular tree.
\end{theorem}

\subsection{Growth}

It is a famous theorem by Gromov that a finitely generated group has polynomial growth if and only if it has a nilpotent subgroup of finite index. We present an analogue for tdlc groups. First we have to define what ``polynomial growth" is supposed to mean for those groups.

\begin{theorem}[\cite{KronMoller2008},Theorem~4.4]\label{thm:polyngrowthdef}
	Let $G$ be a cgtdlc group. The following are equivalent.
	\begin{enumerate}
		\item Let $K \subset G$ be a compact, symmetric generating set for $G$. Let $\mu$ be a Haar measure for $G$. 
		Set $K^n\coloneqq \{g_1 g_2\cdots g_n\mid g_i\in K\}$.
		There are constants $c_1$ and $c_2$ such that $\mu(K^n)\leq c_1 n^{c_2}$ for all natural numbers $n$.
		\item Let $\Gamma$ be a Cayley--Abels graph for $G$ and $\alpha \in \V\Gamma$. There are constants $c_1$ and $c_2$ such that $|\V\B(\alpha,n)|\leq c_1 n^{c_2}$ for all natural numbers $n$.
	\end{enumerate}
\end{theorem}

\begin{exercise}
	Show that the second condition of this theorem is independent of the choice of $\Gamma$.
\end{exercise}

\begin{definition}
	Let $G$ be a cgtdlc group. We say that $G$ has \emph{polynomial growth} if it satisfies the conditions in Theorem \ref{thm:polyngrowthdef}.
\end{definition}

Trofimov and, two years later, Losert prove generalizations of Gromov's result to locally compact groups. Trofimov uses the approach via graph. Losert's result holds for more general locally compact groups, he works with the Haar measures. Note that in the following theorem, once we have the existence of one compact, open, normal subgroup of $G$, the rest follows from Gromov's theorem.

\begin{theorem}[\cite{Trofimov1985} Theorem 2, \cite{Losert87}, Corollary after Theorem 2]
	Let $G$ be a compactly generated, totally disconnected, locally compact
	group.  Then $G$ has polynomial growth if and only if $G$ has a
	compact, open, normal subgroup $K$ such that $G/K$ is a finitely generated almost nilpotent group.
\end{theorem}

\subsection{Valency $2$ and Stalling's end theorem}

We give a characterisation of all non-compact, cgtdlc groups that have a Cayley--Abels graph of valency $2$.
Note that every connected graph of valency $0$ or $1$ is finite.
A group is compact if and only if one, and hence every, Cayley--Abels graph is finite.
The only connected, infinite graph of valency $2$ is the bi-infinite line, its automorphism group is the infinite dihedral group $D_\infty$.

The following theorem about tdlc groups allowing for a $2$-valent Cayley--Abels graph is  a generalization of a well-known statement about finitely generated groups.
For the equivalence of 1., 3. and 4. for a finitely generated group and its Cayley graph, see Hopf \cite[Satz 5]{Hopf1944} and Wall \cite[Lemma 4.1]{Wall1967}.
For cgtdlc groups, Abels showed in \cite[Satz 4.5, Satz 3.10]{a73} that 3. implies 4. and 5.

\begin{theorem}[see \cite{dC18}, Corollary 19.39]\label{T2ends}
	For a cgtdlc group $G$ the following are equivalent.
	\begin{enumerate}
		\item The minimal valency of a Cayley--Abels graph for $G$ is $2$.
		\item The bi-infinite line is a Cayley--Abels graph for $G$.
		\item The group $G$ has precisely two ends.
		\item There is a surjective homomorphism with compact, open kernel from $G$ to the infinite cyclic group or the infinite dihedral group.
		\item The group $G$ has a co-compact, cyclic, discrete subgroup.
	\end{enumerate}
\end{theorem}

\begin{proof}[Sketch of proof]
	Clearly 1. and 2. are equivalent, see explanation at the beginning of this subsection.
	
	 It is also easy to see that 2. is equivalent to 4. Part 2. implies that there exists a continuous homomorphism $G \to D_\infty$ with compact kernel. Because $D_\infty$ is discrete, this kernel is open. It is an exercise to show that every vertex-transitive subgroup of the automorphism group of the bi-infinite line is isomorphic to the infinite cyclic group or the infinite dihedral group.
	
	It is trivial that 2. implies 3.
	
	It is also easy to show that 4. implies 5. Let $\varphi$ be the required homomorphism, then any inverse image of any infinite order element in the image generates a co-compact, cyclic, discrete subgroup.
	
	To prove that 5. implies 3., let $K \subset G$ be a compact subset and $g \in G$ an infinte order element such that $\langle g \rangle K = G$.
	Show that, (using notation from Construction \ref{cons:2}) for any choice of Cayley--Abels graph for $G$ with base $B$ such that $g \in S$, the map $\ZZ \to G/B, \, n \mapsto g^n B$ is a quasi-isometry.
	
	It remains to prove that 3. implies 4. Let $G' \leq G$ be the subgroup fixing both ends of $G$, note that it is an open subgroup of index at most $2$ in $G$. Either we have $G = G'$ or $G = G' \rtimes \ZZ/2\ZZ$.
	It suffices to prove that $G'$ surjects onto $\ZZ$ with compact, open kernel. We use Abels' argument from \cite[Satz 4.5, 2. Fall]{a73}. Let $\Gamma_1,\Gamma_2 \subset \Gamma$ be subgraphs intersecting in finitely many vertices such that $\A\Gamma = \A\Gamma_1 \sqcup \A\Gamma_2$, and each $\Gamma_i$ contains all but finitely many vertices of each representative of one of the two ends.
	Check that the map $G \to \ZZ, \, g \mapsto |\V\Gamma_1 \cap g \V\Gamma_2| - |\V\Gamma_2 \cap g\V\Gamma_1|$ satisfies the claim.
\end{proof}

If a finitely generated group has more than one end, one can say a lot about its structure due to a famous theorem by Stallings \cite{Stallings68} \cite{Stallings71}. There are several proofs available, all use ideas by Dunwoody in a crucial way. A ``short" proof was published in a 9-pages paper by Kr\"on \cite{Kroen10}. It turns out that Stallings end theorem generalizes to cgtdlc groups. First we have go give some definitions.

\begin{definition} \label{amaldef}
	Let $G_1$, $G_2$ and $C$ be arbitrary groups and $\varphi_i\colon C \to G_i$, $i = 1,2$, be group homomorphisms.
	The \emph{amalgam of $G_1$ and $G_2$ with respect to $\varphi_1$ and $\varphi_2$} is a group $G$ together with group homomorphisms
	$\psi_i\colon G_i \to G$, $i = 1,2$, making the following diagram commutative:
	\begin{equation*} \label{amaldiag1}
	\xymatrix{ & G_1 \ar[dr]^{\psi_1} \\ C \ar[ur]^{\varphi_1} \ar[dr]_{\varphi_2} & & G \\ & G_2 \ar[ur]_{\psi_2}
	}
	\end{equation*}
	and satisfying the following universal property:
	For every group $H$ with homomorphisms $\rho_i \colon G_i \to H$, $i = 1,2$, satisfying the above commutative diagram (with $\rho_i$ instead of $\psi_i$),
	there exists a unique homomorphism $\phi \colon G \to H$ such that for $i = 1,2$ the diagram
	\begin{equation*} \label{amaldiag2}
	\xymatrix{ G_i \ar[r]^{\psi_i} \ar[dr]_{\rho_i} & G \ar[d]^{\phi}  \\   & H
	}
	\end{equation*}
	 commutes.
	We denote the amalgam by $G_1 \ast_C G_2$, omitting $\phi_1$ and $\phi_2$ from the notation.
\end{definition}

The amalgamated product often shows up in the context of fundamental groups, as a consequence of the famous theorem of Seifert--van Kampen.
Uniqueness of $G$ is shown via the usual general nonsense argument. Existence is granted by the following construction. Define $N \coloneqq  \lpres \{\phi_1(c)\phi_2(c)^{-1} \mid c \in C\} \rpres \trianglelefteq G_1 \ast G_2$. Here $G_1 \ast G_2$ denotes the free product, which is the amalgated product with respect to the trivial homomorphisms $\{1\} \to G_1$ and $\{1\} \to G_2$, and $\lpres \cdot \rpres$ denotes again the normal closure.
Then $G_1 \ast_C G_2 \cong (G_1 \ast G_2) / N$.

The next definition is that of an \emph{HNN-extension} of a group, named after Higman--Neumann--Neumann. The construction forces two isomorphic subgroups of a group $G$ to be conjugate via a new element.

\begin{definition}
	Let $G$ be a group, let $H \leq G$ be a subgroup and $\alpha \colon H \to G$ an injective homomorphism. The \emph{HNN-extension of $G$ relative to $\alpha$} is the group $G \ast_\alpha \coloneqq  (G \ast \ZZ) / N$, where $N \coloneqq  \lpres\{1_\ZZ h 1_\ZZ^{-1} = \alpha(h) \mid h \in H \} \rpres \trianglelefteq G \ast \ZZ$ and $1_\ZZ \in \ZZ$ denotes the standard generator.
%
\end{definition}

Also HNN-extensions found their importance as fundamental groups, namely of graphs of groups.

The usual end theorem by Stallings is nothing more than its generalization due to Abels \cite[Struktursatz 5.7]{a73} with the additional assumption that $G$ is discrete.
The non-discrete version also has a proof by Kr\"on--M\"oller \cite[Theorem 3.18]{KronMoller2008}.

\begin{theorem}[Stalling's end theorem for cgtdlc groups]
	\label{thm:StallEnd}
	Let $G$ be a cgtdlc group with more than one end.
	Then $G=G_1 \ast_C G_2$ or $G = G_1 \ast_\alpha$ for some compactly generated, open subgropus $G_1,G_2 \lneq G$ and a compact, open subgroup $C \leq G_1,G_2$. In the first case $C \hookrightarrow G_1,G_2$ is the restriction of the identity, in the second case $\alpha \colon C \to G_1$ is a continuous, open injection.
\end{theorem}

\begin{exercise}
	Verify this theorem for $\Aut(\T_d)$.
\end{exercise}

%

\subsection{Free subgroups} 

It is a well-known theorem attributed to Gromov, Stallings, Woess and others that a finitely generated group is quasi-isometric to a regular tree if and only if it contains a non-abelian free group as finite index subgroup.
Analogues of this result for cgtdlc groups were proven by Kr\"on--M\"oller.
The aim of these section is to state these analogues and give the necessary definitions.
All groups that are quasi-isometric to a tree are hyperbolic.

\begin{definition}
	Let $G$ be a locally compact group. A \emph{uniform lattice} in $G$ is a discrete subgroup $D \leq G$ such that the quotient $G/D$ is compact.
\end{definition}

\begin{remark}
	It is a well-known fact that only unimodular groups admit uniform lattices.
	For a reminder on unimodularity, see Section \ref{sect:modular}.
	The reason is that for a closed subgroup $H \leq G$ there exists a $G$-invariant measure on the quotient space $G/H$ if and only if the modular function on $H$ is the restriction of the modular function of $G$ to $H$; and for a discrete subgroup there exists such a measure on the quotient if this quotient is compact.
	This might not be a very satisfying reason because the statements inside are non-trivial to prove.
\end{remark}

\begin{remark}
	Let $G$ be a locally compact group and $H \leq G$ a subgroup such that the quotient $G/H$ is compact. Recall that $G$ is compactly generated if and only if $H$ is.
	Therefore, if a cgtdlc group allows for a cocompact lattice that is isomorphic to a free group, the free group will automatically be finitely generated.
\end{remark}

\begin{theorem}[\cite{KronMoller2008}, Theorem~3.28]
	\label{TQuasiTree}
	Let $G$ be a unimodular, cgtdlc group. Then $G$ is
		quasi-isometric to a regular tree if and only if $G$ has a uniform lattice isomorphic to a free group.
\end{theorem}

As a corollary, Kr\"on--M\"oller obtain that in Stallings' end theorem for cgtdlc groups (Theorem \ref{thm:StallEnd}), $G_1$ and $G_2$ can be taken compact.

\begin{theorem}[\cite{KronMoller2008}, Theorem 3.29]
	Let $G$ be a cgtdlc group quasi-isometric to a regular tree.
	Then $G=G_1 \ast_C G_2$ or $G = G_1 \ast_\alpha$ for some compact, open subgroups $G_1,G_2 \lneq G$ and a compact, open subgroup $C \leq G_1,G_2$. In the first case $C \hookrightarrow G_1,G_2$ is the restriction of the identity, in the second case $\alpha \colon C \to G_1$ is a continuous, open injection.
\end{theorem}

In the situation of the above theorem, Mosher, Sageev and Whyte showed that not only is $G$ quasi-isometric to a tree, it also acts nicely on a tree.
The result can also be found in \cite[Corollary~3.30]{KronMoller2008}

\begin{theorem}[\cite{MSW03} Theorem 9]
	Let $G$ be a cgtdlc group admitting a uniform lattice isomorphic to a  non-abelian free group.
	Then, there exists a locally finite tree $\T$ and an action of $G$ on $\T$ with compact, open vertex stabilizers and compact kernel.
\end{theorem}

There is also a statement without the hypothesis of unimodularity. For completeness we are stating it here, for the definition of graph of groups we refer to Serre's book \cite{s03}.

\begin{theorem}[\cite{KronMoller2008}, Theorem~3.28]
	\label{thm:QItoTreeNonunimod}
	Let $G$ be a cgtdlc group. Then, $G$ is
	quasi-isometric to a tree if and only if $G$ has an expression as a
	fundamental group of a finite graph of groups such that all the vertex and edge
	groups are compact, open subgroups of $G$.
\end{theorem}

\section{Locally compact specialties}
\label{sect:lc}

In this section we talk about concepts defined for tdlc (and sometimes locally compact, Hausdorff in general) groups that are trivial for discrete groups.
The author hopes to convince the reader that Cayley--Abels graphs are more than just a way of applying methods that were originally developed for finitely generated groups.

\subsection{Essential Chief Series}

The content of this subsection is after work of Reid--Wesolek \cite{ReidWesolek18}.
The author recommends the summary paper \cite{Reid18}.

The aim of this section is to ``decompose" a cgtdlc group into pieces that are better understood. It is common in various branches of group theory to apply this strategy via a subnormal series; just think of the composition series for finite groups.

\begin{definition}
	Let $G$ be a tdlc group. An \emph{essentially chief series} of $G$ is a finite series
	\[
	\{1\} = N_0 \triangleleft N_1 \triangleleft \dots \triangleleft N_n = G
	\]
	of closed, normal subgroups of $G$ such that for all $1 \leq i \leq n$ the quotient $N_i/N_{i-1}$ is either
	\begin{itemize}
		\item compact,
		\item discrete, or
		\item a \emph{chief factor} of $G$, i.e., there is no closed, normal subgroup $N \triangleleft G$ with $N_{i-1} \lneq N \lneq N_i$.
	\end{itemize}
\end{definition}

A chief factor does not have to be topologically simple. There might exist a closed, normal subgroup of $N_i$ containing $N_{i-1}$ that is not normal in $G$.
Also, there is no reason why the quotient $N_i/N_{i-1}$ should be compactly generated even if $G$ is.

The aim of this section is to prove the following.

\begin{theorem}[\cite{ReidWesolek18} Lemma 4.3]\label{thm:esschiefseries}
	Let $G$ be a cgtdlc group. Then, $G$ has an essentially chief series.
\end{theorem}

We will need a bit of preparation before going to the proof.

\begin{lemma}\label{lem:discrete_factor}
	Let $G$ be a cgtdlc group and $N \triangleleft G$ a closed, normal subgroup.
	Let $\Gamma$ be a Cayley--Abels graph for $G$ and $K$ the kernel of the action of $G$ on $\Gamma$. Assume that the local action of $N$ on $\Gamma$ is trivial.
	
	Then, $N \cap K$ is a compact, normal subgroup of $G$ and the quotient $N/(N \cap K)$ is discrete.
\end{lemma}

\begin{proof}
	It is clear that the intersection of closed, normal subgroups is closed and normal. In addition it is compact if one of the subgroups is.
	Recall that $N/(N \cap K)$ is discrete if and only if $N \cap K$ is open in $N$.
	Let $\alpha \in \V\Gamma$. Since the local action of $N$ is trivial, we know that $N_{\alpha} \subset N \cap K$. But $N_\alpha$ is open in $N$ and we are done.
\end{proof}

\begin{exercise}\label{exer:quot_CA_degrees}
	Let $N \triangleleft H \triangleleft G$ be closed, normal subgroups of a cgtdlc group $G$. Let $\Gamma$ be a Cayley--Abels graph for $G$ and $\alpha \in \V\Gamma$.
	Show that $\val(N \quot \Gamma) = \val(H \quot \Gamma)$ if and only if
	$N \alpha \cap \B(\alpha,2) = H \alpha \cap \B(\alpha,2)$.
    \emph{Hint:} You can use Proposition \ref{prop:quotient_CA_by_normal}.
\end{exercise}

\begin{proposition}[\cite{ReidWesolek18} Lemma 3.1] \label{prop:for_ess_chief_series}
	Let $G$ be a cgtdlc group and let $\Gamma$ be a Cayley--Abels graph for $G$.
	Let $\mathcal{C}$ be a chain of closed, normal subgroups of $G$, i.e., a set of subgroups totally ordered by inclusion.
	\begin{enumerate}
		\item Let $H \coloneqq  \overline{\bigcup_{N \in \mathcal{C}}N} $. Then $\val(H\quot \Gamma) = \min\{\val(N\quot \Gamma) \mid N \in \mathcal{C} \}$.
		\item Let $H' \coloneqq  \bigcap_{N \in \mathcal{C}} N$. Then $\val(H'\quot \Gamma) = \max\{\val(N\quot \Gamma) \mid N \in \mathcal{C} \}$.
	\end{enumerate}
\end{proposition}

\begin{proof}[Proof sketch via the Chabauty topology.]
	The conditions imply that $H$ and $H'$ are limit points of $\mathcal{C}$, viewed as subspace of the Chabauty space of $G$.
	``The orbit of a given vertex has a given intersection with a given finite set" is a Chabauty-clopen condition.
	Use Exercise \ref{exer:quot_CA_degrees}.
\end{proof}

\begin{proof}[Proof without the Chabauty topology.]
	We first prove 1. Note that $\leq$ is clear. Let $\alpha \in \V\Gamma$.
	We have to find $N \in \mathcal{C}$ such that $\val(H \quot \Gamma) = \val(N \quot \Gamma)$.	
	By Exercise \ref{exer:quot_CA_degrees} this amounts to finding $N \in \mathcal{C}$ with $N\alpha \cap \B(\alpha,2) = H\alpha \cap \B(\alpha,2)$.
	Note that the inclusion $\subset$ is trivial.
	For each $\beta \in H\alpha \cap \B(\alpha,2)$, we find an $N^\beta \in \mathcal{C}$ as follows.
	The set $\{g \in G \mid g \alpha = \beta\}$ is open and has non-empty intersection with $H$. Therefore, it also intersects at least one element of $\mathcal{C}$ non-trivially, i.e., there exists $N^\beta \in \mathcal{C}$ with $\beta \in N^\beta \alpha \cap \B(\alpha,2)$. Since $\mathcal{C}$ is a chain and $\B(\alpha,2)$ has finitely many vertices, we can now take $N$ to be the maximum of $\{N^\beta \mid \beta \in H\alpha \cap  \B(\alpha,2)\}$.
	
	Now we prove 2. This time $\geq$ is clear. Let $\alpha \in \V\Gamma$.
	Again, we have to find $N \in \mathcal{C}$ such that $\val(H' \quot \Gamma) = \val(N \quot \Gamma)$. Again, by Exercise \ref{exer:quot_CA_degrees} this amounts to finding $N \in \mathcal{C}$ with $N\alpha \cap \B(\alpha,2) = H' \alpha \cap \B(\alpha,2)$.
	This time, the inclusion $\supset$ is trivial.
	For each $\beta \in \B(\alpha,2) \setminus H'\alpha$, we find $N^\beta \in \mathcal{C}$ as follows.
	We have
	\[
	\bigcap_{N \in \mathcal{C}} \{n \in N \mid n \alpha = \beta \} = \{h \in H' \mid h\alpha = \beta \} = \emptyset.
	\]
	The sets $\{\{n \in N \mid n \alpha = \beta\} \mid N \in \mathcal{C} \}$ form a chain, totally ordered by inclusion, consisting of compact sets. By Cantor's intersection theorem, one of them has to be empty. Choose $N^\beta \in \mathcal{C}$ such that $\beta \notin N^\beta \alpha$.
	Since $\mathcal{C}$ is a chain and $\B(\alpha,2)$ has finitely many vertices, we can now take $N$ to be the minimum of $\{N^\beta \mid \beta \in  \B(\alpha,2) \setminus H'\alpha\}$.
\end{proof}

\begin{proof}[Proof of Theorem \ref{thm:esschiefseries}]
	We show the theorem by induction on the minimal valency $\operatorname{mval}(G)$ of a Cayley--Abels graph for $G$.
	Assume first that $\operatorname{mval}(G) \leq 1$, then $\operatorname{mval}(G) =0$ and $G$ is compact. If $\operatorname{mval}(G)=2$ we get the result from Theorem \ref{T2ends}. We now assume that all cgtdlc groups $G'$ with $\operatorname{mval}(G') < \operatorname{mval}(G)$ have an essentially chief series.
	
	Let $\Gamma$ be a Cayley--Abels graph for $G$ of smallest possible valency.
	We can apply Zorn's lemma to show that there exists a maximal closed, normal subgroup $H \triangleleft G$ with $\val(H\quot\Gamma) = \val(\Gamma)$.
	Namely, the set of all such closed, normal subgroups is partially ordered by inclusion and Proposition \ref{prop:for_ess_chief_series}(1) grants that every chain has a maximal element. Using Lemma \ref{lem:quotient_by_normal}(4) we see that the local action of $H$ on $\Gamma$ is trivial. By Lemma \ref{lem:discrete_factor} the quotient $H/H \cap K$ is discrete, where $K$ is the kernel of the action of $G$ on $\Gamma$. The quotient $H\quot \Gamma$ is a Cayley--Abels graph for $G/H$ by Proposition \ref{prop:quotient_CA_by_normal}.
	
	Maximality of $H$ implies that for every closed, normal subgroup $N$ of $G$ containing $H$, we have $\val(N\quot \Gamma) < \val(\Gamma)$. By Proposition \ref{prop:for_ess_chief_series}(2), for every chain $\mathcal{C}'$ of closed, normal subgroups lying strictly between $H$ and $G$ we have $\bigcap_{C \in \mathcal{C}'} C \gneq H$, and again by Zorn's lemma, there exists a minimal, closed $H \lneq H' \triangleleft G$. Note that by definition, the quotient $H'/H$ is a chief factor of $G$.
	
	So far we obtained a series of normal subgroups
	\[
	\{1\} \triangleleft H \cap K \triangleleft H \triangleleft H' \triangleleft G,
	\]
	where $H \cap K$ is compact, $H / H \cap K$ is discrete and $H'/H$ is a chief factor of $G$.
	 By induction hypothesis, the group $G/H'$ has an essentially chief series, which can be lifted to complete above series to an essentially chief series for $G$. This finishes the proof.
\end{proof}

Reid--Wesolek also address essentially chief series for more general locally compact groups and discuss uniqueness of of chief factors.

\subsection{Modular function}\label{sect:modular}

Let $G$ be a locally compact group. Recall that the \emph{modular function} is defined as quotient
\[
 \Delta_G \colon G \to \mathbb{R}_+, \quad \Delta_G(g) \coloneqq  \frac{\mu(Ug)}{\mu(U)}
\]
for a left Haar measure $\mu$ on $G$ and an open subset $U \subset G$ with compact closure. It is a homomorphism and independent of the choice of $\mu$ and $U$.
If $G$ is totally disconnected, we can choose $U$ to be a compact, open subgroup of $G$.
Note that for any compact, open subgroup $V \leq U$, all left cosets $gV$ have to have the same measure of $V$. This simple observation, together with additivity of a measure, proves that $\mu(U) = [U:V] \cdot \mu(V)$.
Now left-invariance of the Haar measure shows that $\mu(Ug)=\mu(g^{-1}Ug)$, and we can calculate
\begin{align*}
\Delta_G(g) &= \frac{\mu(Ug)}{\mu(U)} = \frac{\mu(g^{-1}Ug)}{\mu(U)} \\
   &= \frac{\frac{\mu(g^{-1}Ug)}{\mu(U \cap g^{-1}Ug)}}{\frac{\mu(U)}{\mu(U \cap g^{-1}Ug)}} = \frac{[g^{-1}Ug : U \cap g^{-1} U g]}{[U : U \cap g^{-1} U g]}.
\end{align*}
Replacing $U$ with $g^{-1}Ug$ gives us the following lemma.

\begin{lemma}[\cite{Schlichting1979} Lemma 1]\label{lem:modularfct}
	Let $G$ be a tdlc group and $U \leq G$ a compact, open subgroup. Then the modular function is given by $$\Delta_G(g) = \frac{[U : U \cap gUg^{-1}]}{[gUg^{-1} : U \cap gUg^{-1}]}$$ and attains only rational values.
	
	In particular, if $G$ acts transitively on a set $X$ with compact, open point stabilizers, then for every $x \in X$ we have
	\[
	\Delta_G(g) = \frac{|G_{x} (gx)|}{|G_{gx}x|}.
	\]
\end{lemma}

We will now apply this lemma to Cayley--Abels graphs. The idea comes from Bass and Kulkarni \cite[Section 3]{BassKulkarni1990}.
Let $\Gamma$ be a Cayley--Abels graph for $G$. Recall that $\A\Gamma \subset \V\Gamma \times \V\Gamma$ denotes the set of arcs of $\Gamma$. Define $\D_\Gamma \colon \A\Gamma \to \QQ_+$ by
\[
\D_\Gamma(\alpha,\beta) \coloneqq  \frac{|G_\alpha(\beta)|}{|G_\beta(\alpha)|}.
\]
Note that for every $g \in G$ with $g\alpha=\beta$ we have $\D_\Gamma(\alpha,\beta)=\Delta_G(g)$. Iterating this yields the following theorem.

\begin{theorem}[\cite{ALM2022}]
	Let $\Gamma$ be a Cayley--Abels graph for $G$. Let $g \in G$. Let $\alpha \in \V\Gamma$ and let $(\alpha_0,\alpha_1,\dots,\alpha_n)$ be an arc from $\alpha=\alpha_0$ to $\alpha_n = g\alpha$. Then
	\[
	\Delta_G(g) = \D_\Gamma(\alpha_0,\alpha_1) \cdot \D_\Gamma(\alpha_1,\alpha_2) \cdot \dots \cdot \D_\Gamma(\alpha_{n-1},\alpha_n).
	\]
	In particular, it is independent of the chosen arc.
\end{theorem}

\begin{corollary}[\cite{ALM2022}]
	Let $\alpha \in \V\Gamma$ and let $\beta_1,\dots,\beta_d$ be the neighbours of $\alpha$ in $\Gamma$. The image of $\Delta_G$ is generated by $\{\D_\Gamma(\alpha,\beta_1),\dots,\D_\Gamma(\alpha,\beta_d)\}$.
\end{corollary}

The modular function can provide minimal valencies for Cayley--Abels graphs. The following is a special case of a theorem proved in \cite{ALM2022}.

\begin{theorem}[\cite{ALM2022}]
	Let $G$ be a cgtdlc group. Assume that $\Delta_G(G) \leq \mathbb{Z}$ is a cyclic group generated by the rational number $p/q$, where $p, q \in \NN$ are co-prime.
	Then, every Cayley--Abels graph for $G$ has valency at least $p+q$.
\end{theorem}

The second part of the following corollary follows directly from the above theorem, for the first part we refer to Example \ref{expl:arc_graph}.

\begin{corollary}[\cite{ALM2022}]
	Let $\T_{d,d'}$ be a bi-regular tree with valencies $d,d' \geq 2$. Let $G \leq \Aut(\T_{d'})$ be the subgroup leaving the bipartition on the vertices of $\T_{d,d'}$ invariant (note that $G=\Aut(\T_{d,d'})$ unless $d=d'$). Let $\omega$ be an end of $\T_{d,d'}$.
	\begin{enumerate}
		\item The group $G$ has a Cayley--Abels graph of valency $d+d'-2$.
		\item Every Cayley--Abels graph of $G_\omega$ has valency at least $(d-1)(d'-1)+1$.
	\end{enumerate}
\end{corollary}

If $d'=2$, the group $G$ is the automorphism group of the $d$-regular tree, which is a Cayley--Abels graph of the minimal valency $d=d+2-2=(d-1)(2-1)+1$ of both $G$ and $G_\omega$.

\subsection{Scale function}

In this subsection we are working with oriented graphs, also called directed graphs.

\begin{definition}
	Let $\Gamma$ be a graph. An \emph{orientation on $\Gamma$} is a subset $\Or\Gamma \subset \A\Gamma$ such that $\A\Gamma = \Or\Gamma \sqcup i(\Or\Gamma)$.
	In other words, for every edge $\{e,o(e)\}$ exactly one of $e,i(e)$ lies in $\Or\Gamma$.
	Let $\alpha \in \V\Gamma$. The \emph{in-valency} of $\alpha$ is defined as $|t^{-1}(\alpha)|$, the \emph{out-valency} as $|o^{-1}(\alpha)|$.
	A path $(e_1,\dots,e_n)$ in $\Gamma$ is \emph{oriented} if $e_i \in \Or\Gamma$ for all $i=1,\dots,n$.
	A graph morphism $\phi \colon \Gamma \to \Delta$ between two graphs $\Gamma$ and $\Delta$ with orientations $\Or\Gamma$ and $\Or\Delta$ is called \emph{orientation-preserving} if $\phi(\Or\Gamma) \subset \Or\Delta$.
	An action of a group $G$ on $\Gamma$ is orientation-preserving if every element of $G$ acts like an orientation-preserving graph morphism.
\end{definition}

Note that the definition implies that an orientation-preserving graph morphism $\phi$ also satisfies $\phi(\A\Gamma \setminus \Or\Gamma) \subset \A\Delta \setminus \Or\Delta $.

\begin{proposition}[\cite{ALM2022}]\label{Prop:coprime}
	 Let $\Gamma$ be a locally finite graph with orientation $\Or\Gamma$. Let $G$ be a cgtdlc group acting vertex-transitively and orientation-preservingly on $\Gamma$.
	  Assume that $G$ is transitive on $\Or\Gamma$.
	
	If the in- and out-valencies of one (and hence every) vertex are co-prime, then $G$ acts transitively on the set of oriented paths of length $n$ for every $n \geq 0$.
	Moreover, for one (and hence every) vertex $\alpha \in \V\Gamma$, the maximal subgraph of $\Gamma$ containing all oriented paths starting at $\alpha$ (the ``subgraph spanned by the set of descendants of $\alpha$") is a tree.
\end{proposition}

\begin{proof}[Sketch of proof]
	Let $q$ denote the in-valency of $\Gamma$ and $p$ denote the out-valency. 
	Note that, for every $n$, the number of sequences $(e_1,\dots,e_n)$ of arcs such that $o(e_1)=\alpha$, and such that $t(e_{i})=o(e_{i+1})$ but $e_{i}\neq i(e_{i+1})$ for all $i=1,\dots,n-1$ is $p^n$.
	Use Lemma \ref{lem:modularfct} to prove that all those sequences have to have different endpoints and thus they are paths.
\end{proof}

Note that Proposition \ref{Prop:coprime} does not require $\Gamma$ to be connected.
We use it to derive a statement about detecting values of the scale function and finding tidy subgroups.

\begin{corollary}\cite{ALM2022}
	Let \(G\) be a totally disconnected, locally compact group and let $\Gamma$ be a Cayley--Abels graph for $G$.
	Let $g \in G$ and suppose that there exists a vertex $\alpha \in \V\Gamma$ such that $g\alpha$ is a neighbour of $\alpha$.
	
	If $|G_{\alpha}(g\alpha)|$ and $|G_\alpha(g^{-1}\alpha)|$ are co-prime, then $G_\alpha$ is tidy for $g$ and $s_G(g)=|G_{\alpha}(g\alpha)|$.
%
\end{corollary}

\begin{proof}[Sketch of proof]
	Note that the condition implies that $(g\alpha,\alpha) \notin G(\alpha,g\alpha)$.
	Apply Proposition \ref{Prop:coprime} to the subgraph with vertex set $\V\Gamma$, edge set $G\{\alpha,g\alpha\}$ and orientation $G(\alpha,g\alpha)$.
	Use Proposition \ref{prop:scale_power}.
\end{proof}

\subsection{Local prime content}

The local prime content is a local invariant of a tdlc group. It gives the set of primes occurring locally as index between compact, open subgroups.

\begin{definition}\label{def:lpc}
	Let $G$ be a tdlc group. The \emph{local prime content} of $G$ consists of all primes $p \in \NN$ such that every compact, open subgroup $U \leq G$ has a compact, open subgroup $V \leq U$ with $p \mid [U:V]$.
\end{definition}

The following equivalent formulation comes in handy when determining the local prime content in concrete examples.

\begin{lemma}[\cite{w07}, Lemma 2.3]\label{lem:lpc}
	Let $G$ be a tdlc group. Then $p$ is in the local prime content if and only if there exist compact open subgroups $U_0 \geq U_1 \geq U_2 \geq \dots$ of $G$ such that $p \mid [U_n:U_{n+1}]$ for all $n \geq 1.$
\end{lemma}

\begin{exercise}\label{exer:lpc}
	\begin{enumerate}
		\item Show that $V$ can be taken to be a normal subgroup of $U$ in Definition \ref{def:lpc}.
		\item Show that, in Lemma \ref{lem:lpc}, ``for all" can be replaced by ``for infinitely many".
	\end{enumerate}
\end{exercise}

The existence of Cayley--Abels graphs was used by Caprace--Reid--Willis to show that the local prime content of a cgtdlc group without nontrivial, compact, normal subgroups is finite. More precisely, an inspection of their proof reveals the following.

\begin{theorem}[\cite{crw17a}, Proposition 4.6]
	\label{thm:localprimecontentfinite}
	Let $G$ be a non-compact, cgtdlc group and $\Gamma$ a Cayley--Abels graph for $G$ of valency $d$.
	We denote by $K \leq G$ the kernel of the action of $G$ on $\Gamma$.
	Let $L \leq \Sym(d)$ be the local action of $G$ on $\Gamma$.
	Let $p$ be in the local prime content of $G/K$.
	
	Then, $p$ divides the cardinality of a point stabilizer in $L$.
	In particular, the local prime content of $G/K$ is finite.
\end{theorem}

\begin{proof}
	To make notation simpler, we replace $G$ by $G/K$, i.e. we assume that $G \leq \Aut(\Gamma)$.
	
	Let $\alpha_0$ be a vertex of $\Gamma$.
	Let $\Gamma_0$ be the subgraph of $\Gamma$ spanned by $\alpha_0$ and its neighbours.
	We inductively choose vertices $\alpha_1,\alpha_2,\dots$ and define subgraphs $\Gamma_1 \subset \Gamma_2 \subset \dots$ in the following fashion.
	For every $i \geq 1$ the vertex $\alpha_i$ is contained in the subgraph $\Gamma_{i-1}$, but at least one of its neighbours is not.
	Define $\Gamma_i$ to be the subgraph spanned by $\Gamma_{i-1}$ and all neighbours of $\alpha_i$.
	We also require that $\Gamma$ is spanned by $\{\alpha_0,\alpha_1,\dots\}$.
	That these choices are indeed possible is left to the reader as exercise.
	
	Define $U_i \coloneqq  G_{\Gamma_i}$. The condition that $\Gamma$ is spanned by $\{\alpha_0,\alpha_1,\dots\}$ implies that $\bigcap_{i \geq 0} U_i = \{1\}$.
	It now follows from a general fact about profinite groups that the $U_i$ form a neighbourhood basis of the identity in $G$, see Lemma 0.3.1(h) in \cite{Wilson1998}. 
	Clearly $U_i \trianglelefteq U_{i-1}$ for all $i \geq 1$. For all $i \geq 2$ the group $U_{i-1}$ fixes by construction at least one neighbour of $\alpha_i$. Therefore the quotient $U_{i-1}/U_i$ is a subquotient of a point stabilizer in $L$. In particular, all prime divisors of $|U_{i-1}/U_i|$ divide the order of that point stabilizer.
	It is left to prove that $p$ divides the index $[U_{i-1}:U_i]$ for some $i \geq 2$.
	
	By definition there exist compact, open $U \leq U_1$ and $V \leq U$ such that $p \mid [U:V]$. Because the $U_i$ form a neighbourhood basis of the identity we can choose an $n \geq 1$ such that $U_n \leq V$.
	Now the result follows from
	\[
	[U_1 : U_2] \dots [U_{n-1}:U_n] = [U_1:U_n] = [U_1 : U] \cdot [U:V] \cdot [V:U_n]
	\]
	and the fact that if a prime number divides a product, then it divides one of the factors.
%
\end{proof}

The proof also gives a more geometric interpretation of the local prime content.
By Cauchy's theorem, if $p$ divides the finite group $U_{i-1}/U_i$, then this group has an element $g U_i$ of order $p$. The element $g$ will act like an element of order $p$, i.e. like a product of disjoint $p$-cycles, on the neighbouring vertices of $\alpha_i$ that are not contained in $\Gamma_{i-1}$.

\begin{remark}
	It is possible to refine the notion of the local prime content to encapsulate simple groups appearing locally as quotients, and adapt the proof of Theorem \ref{thm:localprimecontentfinite} to get a stronger result. This is done in the paper \cite{ALM2022}.
\end{remark}

Gl\"ockner proved the following connection between the local prime content and the scale function.

\begin{lemma}[\cite{g06}, Proposition 6.2]
	Let $G$ be a tdlc group. Let $g \in G$ and let $p$ be a prime dividing $s_G(g)$. Then, $p$ is in the local prime content of $G$.
\end{lemma}

It has the following consequence, proved by Willis with different methods.

\begin{corollary}[\cite{Willis2001a}]
	Let $G$ be a cgtdlc group.
	The set of all prime divisors of values of the scale function is finite.
\end{corollary}

\bibliographystyle{alpha}
\bibliography{references}
\end{document}